\documentclass{bmcart}

\usepackage{amsthm,amsmath,amsfonts,amssymb}
\usepackage[T1]{fontenc}
\usepackage[utf8]{inputenc} 
\usepackage{eucal}      
\usepackage{mathrsfs}   
\usepackage{longtable}
\usepackage{multirow}
\usepackage{graphicx}
\usepackage{placeins}

\startlocaldefs
\def\Rset{\mathbb{R}}
\def\Cset{\mathbb{C}}
\def\Nset{\mathbb{N}}
\def\Kset{\mathbb{K}}
\def\fsep{\mathrm{sep}(f)}

\def\theenumi{\roman{enumi}}

\theoremstyle{plain}
\newtheorem{theorem}{Theorem}[section]
\newtheorem{lemma}[theorem]{Lemma}
\newtheorem{corollary}[theorem]{Corollary}

\theoremstyle{definition}
\newtheorem{definition}[theorem]{Definition}
\newtheorem{remark}[theorem]{Remark}
\newtheorem{example}[theorem]{Example}
\newtheorem{step}{Step}

\newtheorem*{correctness41}{\bf Proof of the correctness of Definition~4.1}
\newtheorem*{correctness51}{\bf Proof of the correctness of Definition~5.1}

\numberwithin{equation}{section}

\endlocaldefs

\begin{document}

\begin{frontmatter}

\begin{fmbox}
\dochead{Research}

\title{On the convergence of high-order Ehrlich-type iterative methods for approximating all zeros of a polynomial simultaneously}

\author[
   addressref={aff1},
   corref={aff1},
   email={proinov@uni-plovdiv.bg}
]{\inits{PD}\fnm{Petko D} \snm{Proinov}}
\author[
   addressref={aff1},
   email={mariavasileva@uni-plovdiv.bg}
]{\inits{MT}\fnm{Maria T.} \snm{Vasileva}}

\address[id=aff1]{%
  \orgname{Faculty of Mathematics and Informatics, University of Plovdiv},
  \postcode{4000}
  \city{Plovdiv},
  \cny{Bulgaria}
}

\end{fmbox}

\begin{abstractbox}

\begin{abstract}
We study a family of high order Ehrlich-type methods for approximating all zeros of a polynomial simultaneously.
Let us denote by $T^{(1)}$ the famous Ehrlich method (1967).
Starting from $T^{(1)}$, Kjurkchiev and Andreev (1987) have introduced recursively a sequence ${(T^{(N)})_{N = 1}^\infty}$
of iterative methods for simultaneous finding polynomial zeros.
For given $N \ge 1$, the Ehrlich-type method $T^{(N)}$ has the order of convergence ${2 N + 1}$.
In this paper, we establish two new local convergence theorems as well as a semilocal convergence theorem 
(under computationally verifiable initial conditions and with a posteriori error estimate) for the Ehrlich-type methods $T^{(N)}$.
Our first local convergence theorem generalizes a result of Proinov (2015) and improves the result of Kjurkchiev and Andreev (1987).
The second local convergence theorem generalizes another recent result of Proinov (2015), but only in the case of maximum-norm.
Our semilocal convergence theorem is the first result in this direction.

\end{abstract}



\end{abstractbox}

\end{frontmatter}

%
%

\section{Introduction}
\label{sec:introduction}

Throughout this paper ${(\Kset,|\cdot|)}$ denotes an algebraically closed field and
$\Kset[z]$ denotes the ring of polynomials (in one variable) over $\Kset$.
For a given vector $x$ in $\Kset^n$, ${x_i}$ always denotes the $i$th coordinate of $x$.
In particular, if $F$ is a map with values in $\Kset^n$, then ${F_i(x)}$ denotes the $i$th coordinate of the vector $F(x)$.
We endow the vector space $\Kset^n$ with a norm $\|x\|_p$ defined as usual:
\[
\|x\|_p = \left( \sum _{i = 1} ^n |x_i|^p \right) ^{1/p} \qquad (1 \le p \le \infty)
\]
and with coordinate-wise ordering $\preceq$ defined by
\begin{equation} \label{eq:coordinate-wise-ordering}
x \preceq y  \quad\text{if and only if}\quad x_i \le y_i \,\, \text{ for each } \,\, i = 1, \ldots, n
\end{equation}
for ${x, y \in \Rset^n}$. Then ${(\Rset^n,\|\cdot\|_p)}$ is a solid vector space.
Also, we endow $\Kset^n$ with the cone norm ${\|\cdot\| \colon \Kset^n \to \Rset^n }$ defined by
\[
\|x\| = (|x_1|,\ldots,|x_n|).
\]
Then ${(\Kset^n,\|\cdot\|, \preceq)}$ is a cone normed space over $\Rset^n$ (see, e.g., Proinov \cite{Pro13}).

Let ${f \in \Kset[z]}$ be a polynomial of degree ${n \ge 2}$.
A vector ${\xi \in \Kset^n}$ is said to be a root-vector of $f$ if
${f(z) = a_0 \prod _{i = 1}^ n (z - \xi_i)}$ for all ${z \in \Kset}$,
where ${a_0 \in \Kset}$.
We denote with $\fsep$  the \emph{separation number} of $f$ which is defined to be the minimum distance between two distinct zeros of $f$.


\subsection{The Weierstrass method and Weierstrass correction}
\label{subsec:Weierstrass-method}

In the literature, there are a lot of iterative methods for finding all zeros of $f$ simultaneously (see, e.g., the monographs of
Sendov, Andreev and Kjurkchiev \cite{SAK94},
Kjurkchiev \cite{Kyu98},
McNamee \cite{McN07} and
Petkovi\'c \cite{Pet08} and references given therein).
In 1891, Weierstrass \cite{Wei91} published his famous iterative method for simultaneous computation of all zeros of $f$.
The \emph{Weierstrass method} is defined by the following iteration
\begin{equation} \label{eq:Weierstrass-iteration}
x^{(k + 1)}  = x^{(k)} - W_f(x^{(k)}), \qquad k = 0, 1, 2,\ldots,
\end{equation}
where the operator ${W_f \colon \mathcal{D} \subset \Kset^n \to \Kset^n}$ is defined by
\begin{equation} \label{eq:Weierstrass-correction}
W_f(x) = (W_1(x),\ldots,W_n(x)) \quad\text{with}\quad
W_i(x) = \frac{f(x_i)}{ a_0 \prod_{j \ne i} {(x_i  - x_j )}} \quad (i = 1,\ldots,n),
\end{equation}
where ${a_0 \in \Kset}$ is the leading coefficient of $f$ and the domain $\mathcal{D}$ of $W$ is the set of all vectors in $\Kset^n$ with distinct components.
The Weierstrass method \eqref{eq:Weierstrass-iteration} has second-order of convergence provided that all zeros of $f$ are simple.
The operator $W_f$ is called \emph{Weierstrass correction}. We should note that $W_f$ plays an important role in many semilocal convergence theorems for simultaneous methods.


\subsection{The Ehrlich method}
\label{subsec:Ehrlich-method}

Another famous iterative method for finding simultaneously all zeros of a polynomial $f$ was introduced by Ehrlich \cite{Ehr67} in 1967.
The \emph{Ehrlich method} is defined by the following fixed point iteration:
\begin{equation} \label{eq:Ehrlich-method}
x^{(k + 1)} = T(x^{(k)}), \qquad k = 0, 1, 2, \ldots,
\end{equation}
where the operator ${T \colon \mathscr{D} \subset \Kset^n \to \Kset^n}$ is defined by ${T(x) = (T_1(x),\ldots,T_n(x))}$ with
\begin{equation} \label{eq:Ehrlich-correction}
T_i(x) = x_i - \frac{f(x_i)}{f'(x_i) - f(x_i) \displaystyle\sum_{j \ne i} \frac{1}{x_i - x_j}} \qquad (i = 1,\ldots,n)
\end{equation}
and the domain of $T$ is the set
\begin{equation} \label{eq:domain-Ehrlich-method}
\mathscr{D} =
\{ x \in  \mathcal{D} \colon f'(x_i) - f(x_i) \displaystyle \sum_{j \ne i} \frac{1}{x_i - x_j} \ne 0 \, \text{ for } i \in I_n \}.
\end{equation}
Here and throughout the paper, we  denote by $I_n$ the set of indices ${1, \ldots, n}$, that is ${I_n = \{1, \ldots, n\}}$.
The Ehrlich method has third-order of convergence if all zeros of $f$ are simple.
The Ehrlich method was rediscovered by Abert \cite{Abe73} in 1973.
In 1970, B\"orsch-Supan \cite{Bor70} introduced another third-order method for numerical computation of all zeros of a polynomial simultaneously.
In 1982, Werner \cite{Wer82} has proved that the both methods are identical.
The Ehrlich method \eqref{eq:Ehrlich-method} is known also as ``Ehrlich-Abert method'', ``B\"orsch-Supan method'' and ``Abert method''.

Recently, Proinov \cite{Pro15c} obtained two local convergence theorems for Ehrlich method under different types of initial conditions. 
The first one generalizes and improves the results of
Kyurkchiev and Tashev \cite{KT81,TK83} and
Wang and Zhao \cite[Theorem~2.1]{WZ89}.
The second one generalizes and improves the results of
Wang and Zhao \cite[Theorem~2.2]{WZ89}
and Tilli \cite[Theorem~3.3]{Til98}.

Before we state the two results of \cite{Pro15c}, we need some notations which will be used throughout the paper.
For given vectors ${x \in \Kset^n}$ and ${y \in \Rset^n}$, we define in ${\Rset^n}$ the vector
\[
\frac{x}{y} = \left( \frac{|x_1|}{y_1},\ldots,\frac{|x_n|}{y_n} \right),
\]
provided that $y$ has no zero components. Given $p$ such that ${1 \le p \le \infty}$,
we always denote by $q$ the conjugate exponent of $p$, i.e. $q$ is defined by means of
\[
1 \le q \le \infty \quad\text{and}\quad 1/p + 1/q = 1.
\]
In the sequel, we use the function ${d \colon \Kset^n \to \Rset^n}$ defined by ${d(x) = (d_1(x),\ldots,d_n(x))}$ with
\[
d_i(x) = \min_{j \ne i} |x_i - x_j| \qquad (i = 1,\ldots,n).
\]
Let ${a > 0}$ and ${b \ge 1}$. We define the real function $\phi$ by
\begin{equation} \label{eq:phi-Proinov}
  \phi(t) = \frac{a t^2}{(1 - t) (1 - b t)  - a t^2}
\end{equation}
and the real number $R$ as follows
\begin{equation} \label{eq:R-Proinov}
R = \frac{2}{b + 1 + \sqrt{(b - 1)^2 + 8a}} \, .
\end{equation}

\begin{theorem}[Proinov \cite{Pro15c}]  \label{thm:Proinov-Ehrlich-local-first}
Let ${f \in \Kset[z]}$ be a polynomial of degree ${n \ge 2}$
which has only simple zeros,
${\xi \in \Kset^n}$ be a root-vector of $f$ and ${1 \le p \le \infty}$.
Suppose ${x^{(0)} \in \Kset^n}$ is an initial guess satisfying
\begin{equation} \label{eq:Proinov-initial-condition-first-local}
E(x^{(0)}) = \left \| \frac{x^{(0)} - \xi}{d(\xi)} \right \|_p < R = \frac{2}{b + 1 + \sqrt{(b - 1)^2 + 8a}},
\end{equation}
where  ${a = (n - 1) ^{1/q}}$ and ${b = 2^ {1/q}}$.
Then Ehrlich iteration \eqref{eq:Ehrlich-method} is well-defined and converges cubically to $\xi$ with error estimates
\[
 \|x^{(k + 1)} - \xi\| \preceq \lambda^{3^k} \|x^{(k)} - \xi\|
 \quad\text{and}\quad
  \|x^{(k)} - \xi\| \preceq \lambda^{(3^k - 1)/2} \|x^{(0)} - \xi\|
\]
for all ${k \ge 0}$, where ${\lambda = \phi(E(x^{(0)}))}$ and the function $\phi$ is defined by \eqref{eq:phi-Proinov}.
\end{theorem}

\begin{theorem}[Proinov \cite{Pro15c}]  \label{thm:Proinov-Ehrlich-local-second}
Let ${f \in \Kset[z]}$ be a polynomial of degree ${n \ge 2}$,
${\xi \in \Kset^n}$ be a root-vector of $f$ and ${1 \le p \le \infty}$.
Suppose ${x^{(0)} \in \Kset^n}$ is a vector with distinct components satisfying
\begin{equation} \label{eq:Proinov-initial-condition-second-local}
E(x^{(0)}) = \left \| \frac{x^{(0)} - \xi}{d(x^{(0)})} \right \|_p \le R = \frac{2}{b + 1 + \sqrt{(b - 1)^2 + 8a}},
\end{equation}
where ${a = (n - 1) ^{1/q}}$ and ${b = 2^ {1/q}}$.
Then $f$ has only simple zeros in $\Kset$ and Ehrlich iteration \eqref{eq:Ehrlich-method} is well-defined and converges to $\xi$ with error estimates
\[
 \|x^{(k + 1)} - \xi\| \preceq \theta \lambda^{3^k} \|x^{(k)} - \xi\|
 \quad\text{and}\quad
  \|x^{(k)} - \xi\| \preceq \theta^k \lambda^{(3^k - 1)/2} \|x^{(0)} - \xi\|
\]
for all ${k \ge 0}$, where ${\lambda = \phi(E(x^{(0)}))}$, ${\theta = \psi(E(x^{(0)}))}$ and the function $\phi$ is defined by
\eqref{eq:phi-Proinov} and the function $\psi$  by
\begin{equation} \label{eq:psi-Proinov}
   \psi(t) = \frac{(1 - t) (1 - b t)  - a t^2}{1 - t - a t^2}  \, .
\end{equation}
Moreover, the method converges cubically to $\xi$ provided that ${E(x^{(0)}) < R}$. 	
\end{theorem}


\subsection{A family of high-order Ehrlich-type methods}
\label{subsec:High-order-Ehrlich-type-methods}

In the following definition, we define a sequence $(T^{(N)})_{N = 0}^\infty$ of iteration functions in the vector space $\Kset^n$.
In what follows, we define the binary relation $\#$ on $\Kset^n$ by
\begin{equation} \label{eq:special-binary-relation}
	x \, \# \, y  \quad\Leftrightarrow\quad  x_i \ne y_j \, \text{ for all } \,  i,j \in I_n \, \text{ with } \,  i \ne j.
\end{equation}

\begin{definition} \label{def:iteration-functions}
Let ${f \in \Kset[z]}$ be a polynomial of degree ${n \ge 2}$.
Define the sequence $(T^{(N)})_{N=0}^\infty$ of functions ${T^{(N)} \colon D_N \subset \Kset^n \to \Kset^n}$ recursively
by setting ${T^{(0)}(x) = x}$ and
\begin{equation} \label{eq:Ehrlich-high-order-iteration-function}
T_i^{(N + 1)} (x) = x_i - \frac{f(x_i)}{f'(x_i) - f(x_i) \displaystyle\sum_{j \ne i} \frac{1}{x_i - T_j^{(N)}(x)}} \quad (i = 1,\ldots,n),
\end{equation}
where the sequence of domains $(D_N)_{N = 0}^\infty$ is also defined recursively by setting ${D_0 = \Kset^n}$ and
\begin{equation} \label{eq:domain-Ehrlich-high-order}
D_{N + 1}  =  \{ x \in D_N : x \, \# \,T^{(N)}(x), \,
f'(x_i) - f(x_i) \displaystyle \sum_{j \ne i} \frac{1}{x_i - T_j^{(N)}(x)} \ne 0 \text{ for } i \in I_n \}.
\end{equation}
\end{definition}

Given ${N \in \Nset}$, the $N$th method of Kjurkchiev-Andreev's family can be defined by the following fixed-point iteration:
\begin{equation} \label{eq:Ehrlich-high-order}
x^{(k + 1)} = T^{(N)} (x^{(k)}), \qquad k = 0, 1, 2, \ldots.
\end{equation}

It is easy to see that in the case ${N = 1}$ the Ehrlich-type method \eqref{eq:Ehrlich-high-order} coincides
with the classical Ehrlich method \eqref{eq:Ehrlich-method}.
The order of convergence of the Ehrlich-type method \eqref{eq:Ehrlich-high-order} is ${2 N + 1}$.

Kjurkchiev and Andreev \cite{KA87} established the following convergence result for the Ehrlich-type methods \eqref{eq:Ehrlich-high-order}.
This result and its proof can also be found in the monographs of
Sendov, Andreev and Kjurkchiev \cite[Section~19]{SAK94} and Kjurkchiev \cite[Chapter 9.2]{Kyu98}).

\begin{theorem}[Kjurkchiev and Andreev \cite{KA87}] \label{thm:Kjurkchiev-Andreev}
Let ${f \in \Cset[z]}$ be a polynomial of degree ${n \ge 2}$ which has only simple zeros, ${\xi \in \Cset^n}$ be a root-vector of $f$ and ${N \ge 1}$.
Let ${0 < h < 1}$ and ${c > 0}$ be such that
\begin{equation} \label{eq:Kjurkchiev-Andreev-h-c}
\delta > 2 c (1 + (2 n - 1) h) \, \text{ and } \, \frac{n c^2}{(\delta - c) (\delta - 2 c - 2 c h) - 3 (n - 1) c^2 h^2} \le 1 ,
\end{equation}
where ${\delta = \fsep}$. Suppose ${x^{(0)} \in \Cset^n}$ is an initial guess satisfying the condition
\begin{equation} \label{eq:Kjurkchiev-Andreev-initial-condition}
\|x^{(0)} - \xi\|_\infty \le c h.
\end{equation}
Then the Ehrlich-type method \eqref{eq:Ehrlich-high-order} converges to $\xi$ with error estimate
\begin{equation} \label{eq:Kjurkchiev-Andreev-error-estimate}
 \|x^{(k)} - \xi\|_\infty \le c h ^{(2 N + 1)^k} \quad\text{for all } \,  k \ge 0.
\end{equation}
\end{theorem}

\subsection{The purpose of the paper}
\label{subsec:The-purpose-of-the-paper}

In this paper, we present two new local convergence theorems as well as a semilocal convergence theorem (under computationally verifiable initial conditions and with a posteriori error estimate) for Ehrlich-type methods \eqref{eq:Ehrlich-high-order}.
Our first local convergence result (Theorem~\ref{thm:first-local-Ehrlich}) generalizes Theorem~\ref{thm:Proinov-Ehrlich-local-first}
(Proinov \cite{Pro15c}) and improves Theorem~\ref{thm:Kjurkchiev-Andreev} (Kjurkchiev and Andreev \cite{KA87}).
Our second local convergence result (Theorem~\ref{thm:second-local-Ehrlich}) generalizes Theorem~\ref{thm:Proinov-Ehrlich-local-second}
(Proinov \cite{Pro15c}), but only in the case ${p = \infty}$.
Furthermore, several numerical examples are provided to show some practical applications of our semilocal convergence result.


\section{A general convergence theorem}
\label{sec:A-general-convergence-theorem}

Recently, Proinov \cite{Pro09,Pro10,Pro15a} has developed a general convergence theory for iterative processes of the type
\begin{equation} \label{eq:Picard-sequence}
    x_{k+1} = Tx_k, \qquad k = 0,1,2,\ldots,
\end{equation}
where ${T \colon D \subset X \to X}$ is an iteration function in a cone metric space $X$. 
In order to make this paper self-contained, we briefly review some basic definitions and results from this theory.

Throughout this paper $J$ denotes an interval on ${\Rset_+}$ containing $0$.
For an integer ${k \ge 1}$, we denote by ${S_k(t)}$ the following polynomial: 
\[
{S_k(t) = 1 + t + \ldots + t^{k-1}}.
\]
If ${k = 0}$ we assume that ${S_k(t) \equiv 0}$. Throughout the paper we assume by definition that ${0^0 = 1}$.

\begin{definition}[\cite{Pro10}] \label{def:quasi-homogeneous-function}
A function ${\varphi \colon J \to \Rset_+}$ is called \emph{quasi-homogeneous} of degree ${r \ge 0}$ on $J$
if it satisfies the following condition:
\begin{equation} \label{eq:QHF1}
\varphi(\lambda t) \le \lambda^r \varphi(t)
\quad\text{for all } \lambda  \in [0,1] \text{ and } t  \in J.
\end{equation}
\end{definition}

If $m$ functions ${\varphi_1,\ldots,\varphi_m}$ are quasi-homogeneous on $J$ of degree ${r_1,\ldots,r_m}$, 
then their product ${\varphi_1 \ldots \varphi_m}$ is a quasi-homogeneous function of degree ${r_1 + \ldots + r_m}$ on $J$. 
Note also that a function $\varphi$ is quasi-homogeneous of degree $0$ on $J$ if and only it is nondecreasing on $J$.
  
\begin{definition}[\cite{Pro09}] \label{def:gauge-function-high-order}
A function ${\varphi \colon J \to J}$ is said to be a
\emph{gauge function of order} ${r \ge 1}$ on $J$ if it satisfies the following conditions:
\begin{enumerate}
  \item $\varphi$ is quasi-homogeneous of degree $r$ on $J$;
  \item ${\varphi(t) \le t}$ \quad for all ${t \in J}$.
\end{enumerate}
A gauge function $\varphi$ of order $r$ on $J$ is said to be a \emph{strict gauge function} if
the inequality in (ii) holds strictly whenever ${t \in J \backslash \{ 0 \}}$.
\end{definition}

The following is a sufficient condition for a gauge function of order $r$.

\begin{lemma}[\cite{Pro10}] \label{lem:gauge-function-condition}
If ${\varphi \colon J \to \Rset_+}$ is a quasi-homogeneous function of degree ${r \ge 1}$ on an interval $J$ and ${R > 0}$ is a fixed point of $\varphi$ in $J$, then $\varphi$ is a gauge function of order $r$ on ${[0, R]}$. Moreover, if ${r > 1}$, then  function $\varphi$ is a strict gauge of order $r$ on ${J = [0, R)}$.
\end{lemma}

\begin{definition}[\cite{Pro09}] \label{def:function-of-initial-conditions}
Let ${T \colon D \subset X  \to X}$ be a map on an arbitrary set $X$. A function ${E \colon D \to \Rset_+}$ is said to be a \emph{function of initial conditions} of $T$ (with a gauge function $\varphi$ on $J$) if there exist a function ${\varphi \colon J \to J}$ such that
\begin{equation} \label{eq:function-of-initial-conditions}
    E(Tx) \le \varphi(E(x)) \quad \text{ for all } x \in D  \text{ with } Tx \in D \text{ and } E(x) \in J.
\end{equation}
\end{definition}

\begin{definition}[\cite{Pro09}] \label{def:initial-points}
Let ${T \colon D \subset X \to X}$ be a map on an arbitrary set $X$, and let ${E \colon D \to \Rset_+}$ be a function of initial conditions of $T$ with a gauge function on $J$. Then a point ${x \in D}$ is said to be an \emph{initial point} of $T$ (with respect to $E$) if ${E(x) \in J}$ and all of the iterates ${T^{k}x}$  ${(k = 0,1,2,\ldots)}$ are well-defined and belong to $D$.
\end{definition}

The following is a simple sufficient condition for initial points.

\begin{theorem}[{\cite{Pro10}}]  \label{thm:initial-point-test}
Let ${T \colon D \subset X  \to X}$ be a map on an arbitrary set $X$ and ${E \colon D \to \Rset_+}$ be a function of
initial conditions of $T$ with a gauge function $\varphi$ on $J$.
Suppose that
${x \in D}$ with ${E(x) \in J}$ implies ${Tx \in D}$.
Then every point ${x_0 \in D}$ such that ${E(x_0) \in J}$ is an initial points of $T$.
\end{theorem}

\begin{definition}[\cite{Pro15a}] \label{def:iterated-contraction-map}
Let ${T \colon D \subset X \to X}$ be an operator in a cone normed space ${(X,\|\cdot\|)}$ over a solid
vector space ${(Y,\preceq)}$, and let ${E \colon D \to \Rset_+}$ be a function of initial conditions of $T$ with a gauge
function on an interval $J$. Then the operator $T$ is said to be \emph{an iterated contraction at a point} ${\xi \in D}$ 
(with respect to $E$) if $E{(\xi) \in J}$ and
\begin{equation} \label{eq:iterated-contraction-map}
\|Tx - \xi \| \preceq \beta(E(x)) \| x - \xi\| \quad\text{for all } x \in D  \text{ with } E(x) \in J,
\end{equation}
where the control function ${\beta \colon J \to [0,1)}$ is nondecreasing.
\end{definition}

The following fixed point theorem plays an important role in our paper. 

\begin{theorem} [Proinov \cite{Pro15a}] \label{thm:general-convergence-theorem}
Let ${T \colon D \subset X \to X}$ be an operator of a cone normed space ${(X,\|\cdot\|)}$ over a solid vector space ${(Y,\preceq)}$, 
and let ${E \colon D \to \Rset_+}$ be a function of initial conditions of $T$ with a gauge function $\varphi$ of order ${r \ge 1}$ on an interval $J$.
Suppose $T$ is an iterated contraction at a point $\xi$ with respect to $E$ 
with control function $\beta$ such that
\begin{equation} \label{eq:Proinov-beta-gauge-property}
  t \, \beta(t) \text{ is a strict gauge function of order } r \text{ on } J
\end{equation}
and there exist a function ${\psi \colon J \to \Rset_+}$ such that
\begin{equation} \label{eq:Proinov-beta-property}
\beta(t) = \phi(t) \psi(t) \quad\text{for all }t \in J,
\end{equation}
where ${\phi \colon J \to \Rset_+}$ is a nondecreasing function satisfying
\begin{equation} \label{eq:Proinov-phi-property}
\varphi(t) = t \, \phi(t) \quad\text{for all } t  \in J.
\end{equation}
Then the following statements hold true.
\begin{enumerate}
  \item The point $\xi$ is a unique fixed point of $T$ in the set ${U = \left\{ x \in D : E(x) \in J \right\}}$.
  \item Starting from each initial point ${x^{(0)}}$ of $T$, Picard iteration \eqref{eq:Picard-sequence} remains in the set $U$ and converges to $\xi$ with error estimates
\begin{equation} \label{eq:general-con-theorem-error-estimates}
\|x^{(k+1)} - \xi\| \preceq \theta \lambda^{r^k} \, \|x^{(k)} - \xi\|
\quad\text{and}\quad
 \|x^{(k)} - \xi\| \preceq \theta^{k} \lambda^{S_k(r)} \, \|x^{(0)} - \xi\|
\end{equation}
for all ${k \ge 0}$, where ${\lambda = \phi(E(x^{(0)}))}$ and ${\theta = \psi(E(x^{(0)}))}$.
\end{enumerate}
\end{theorem}

In the case ${\beta \equiv \phi}$, Theorem~\ref{thm:general-convergence-theorem} reduces to the following
result.

\begin{corollary} [\cite{Pro15a}] \label{cor:general-convergence-theorem}
Let ${T \colon D \subset X \to X}$ be an operator in a cone normed space ${(X,\|\cdot\|)}$ over a solid
vector space ${(Y,\preceq)}$, and let ${E \colon D \to \Rset_+}$ be a function of initial conditions of $T$ with a strict gauge
function $\varphi$ of order ${r \ge 1}$ on an interval $J$.
Suppose that $T$ is an iterated contraction at a point $\xi$ with respect to $E$ and
with control function $\phi$ satisfying \eqref{eq:Proinov-phi-property}.
Then the following statements hold true.
\begin{enumerate}
  \item The point $\xi$ is a unique fixed point of $T$ in the set ${U = \left\{ x \in D : E(x) \in J \right\}}$.
  \item Starting from each initial point ${x^{(0)}}$ of $T$, Picard iteration \eqref{eq:Picard-sequence} remains in $U$ 
	and converges to $\xi$ with order $r$ and error estimates
\begin{equation} \label{eq:general-con-theorem-error-estimates-phi}
\|x^{(k+1)} - \xi\| \preceq \lambda^{r^k} \, \|x^{(k)} - \xi\|
\quad\text{and}\quad
 \|x^{(k)} - \xi\| \preceq  \lambda^{S_k(r)} \, \|x^{(0)} - \xi\|,
\end{equation}
for all ${k \ge 0}$, where ${\lambda = \phi(E(x^{(0)}))}$.
\end{enumerate}
\end{corollary}


\section{Some inequalities in $\Kset^n$}
\label{sec:Some-inequalities}

In this section, we present some useful inequalities in $\Kset^n$ which play an important role in the paper.

\begin{lemma} [\cite{PC14}] \label{lem:u-v-inequalities-basic}
Let ${u,v \in \Kset^n}$, $v$  be a vector with distinct components and ${1 \le p \le \infty}$.
Then for all ${i,j \in I_n}$,
\begin{equation} \label{eq:u-v-inequalities-basic-1}
    |u_i - v_j| \ge \left(1 - \left\| \frac{u - v}{d(v)}\right\|_p \right) |v_i - v_j|,
\end{equation}
\begin{equation} \label{eq:u-v-inequalities-basic-2}
|u_i - u_j| \ge \left(1 - \displaystyle  2^ {1/q} \left \| \frac{u - v}{d(v)} \right\|_p \right) |v_i - v_j|.	
\end{equation}
\end{lemma}

\begin{lemma} [\cite{Pro15a}] \label{lem:u-v-inequality-distinct-components}
Let ${u,v \in \Kset^n}$ and ${1 \le p \le \infty}$. If the vector
$v$ has distinct components  and
\[
\left\| \frac{u - v}{d(v)}\right\|_p < \frac{1}{2}
\]
then the vector $u$ also has distinct components.
\end{lemma}

\begin{lemma} [\cite{PV15}]  \label{lem:u-v-inequalities-1}
Let ${u,v,\xi \in \Kset^n}$, $\xi$ be a vector with distinct components, ${1 \le p \le \infty}$ and
\begin{equation} \label{eq:inequality-vectors-cone-norm}
    \|v - \xi \| \preceq  \|u - \xi\|.
\end{equation}
Then for all ${i,j \in I_n}$,
\begin{equation} \label{eq:u-v-inequalities-1}
    |u_i - v_j| \ge \left( 1 - 2^ {1/q} \left\| \frac{u - \xi}{d(\xi)} \right\|_p \right) |\xi_i - \xi_j|.
\end{equation}
\end{lemma}

\begin{lemma} \label{lem:u-v-inequalities-2}
Let ${u,v,\xi \in \Kset^n}$, ${\alpha \ge 0}$ and ${1 \le p \le \infty}$. If $v$ is a vector with distinct components such that
\begin{equation} \label{eq:inequality-vectors-cone-norm-h-condition}
    \|u - \xi\| \preceq \alpha \|v - \xi\|,
\end{equation}
then for all ${i,j \in I_n}$,
\begin{equation} \label{eq:u-v-inequalities-2}
    |u_j - v_i| \ge \left( 1 - (1 + \alpha) \left\| \frac{v - \xi}{d(v)} \right\|_p \right) |v_i - v_j|.
\end{equation}
\end{lemma}

\begin{proof}
By the triangle inequality of cone norm in $\Kset^n$  and \eqref{eq:inequality-vectors-cone-norm-h-condition}, we obtain
\[
\|u - v\| \preceq \|u - \xi\| + \|v - \xi\| \preceq (1 + \alpha) \|v - \xi\|,
\]
which yields
\[
 \left \| \frac{u - v}{d(v)} \right\| \preceq (1 + \alpha) \left \| \frac{v - \xi}{d(v)} \right\|.
\]
Taking the $p$-norm, we get
\begin{equation} \label{eq:u-v-h-norm-inequality}
 \left \| \frac{u - v}{d(v)} \right\|_p \le (1 + \alpha) \left \| \frac{v - \xi}{d(v)} \right\|_p.
\end{equation}
From \eqref{eq:u-v-inequalities-basic-1} and \eqref{eq:u-v-h-norm-inequality}, we obtain \eqref{eq:u-v-inequalities-2}
which completes the proof.
\end{proof}

\begin{lemma} \label{lem:u-v-inequalities-3}
Let ${u,v,\xi \in \Kset^n}$, ${\alpha \ge 0}$ and ${1 \le p \le \infty}$. If $v$ is a vector with distinct components such that \eqref{eq:inequality-vectors-cone-norm-h-condition} holds,
then for all ${i,j \in I_n}$,
\begin{equation} \label{eq:u-v-inequalities-3}
    |u_i - u_j| \ge \left(1 - 2^{1/q} \, (1 + \alpha) \left\| \frac{v - \xi}{d(v)} \right\|_p  \right) |v_i - v_j|.
\end{equation}
\end{lemma}

\begin{proof}
From \eqref{eq:u-v-inequalities-basic-2} and \eqref{eq:u-v-h-norm-inequality}, we get \eqref{eq:u-v-inequalities-3}
which completes the proof.
\end{proof}


\section{Local convergence theorem of the first type}
\label{sec:Local-convergence-theorem-of-the-first-type}

Let ${f \in \Kset[z]}$ be a polynomial of degree ${n \ge 2}$ which has only simple zeros in $\Kset$,
and let  ${\xi \in \Kset^n}$ be a root-vector of $f$.
In this section we study the convergence of the Ehrlich-type methods \eqref{eq:Ehrlich-high-order} with respect to the function of initial conditions
${E \colon \Kset^n \to \Rset_+}$ defined as follows
\begin{equation} \label{eq:FIC-first}
    E(x) = \left\| \frac{x - \xi}{d(\xi)}  \right\|_p \quad (1 \le p \le \infty).
\end{equation}

Let ${a > 0}$ and ${b \ge 1}$. Throughout this section, we define the function $\phi$ and the real number $R$ by
\eqref{eq:phi-Proinov} and \eqref{eq:R-Proinov}, respectively.
It is easy to show that $R$ is the unique solution of the equation
${\phi(t) = 1}$ in the interval ${[0,\tau)}$, where ${\tau = 2 / (b + 1 + \sqrt{(b - 1)^2 + 4 a})}$.
Note that $\phi$ is an increasing function which maps $[0,R]$ onto [0,1].
Besides, $\phi$ is quasi-homogeneous of degree $2$ on ${[0,R]}$.
In the next definition, we introduce a sequence of such functions.

\begin{definition} \label{def:phi-N}
We define the sequence $(\phi_N)_{N = 0}^\infty$ of nondecreasing functions $\phi_N \colon [0,R] \to [0,1]$ 
recursively by setting
$ \phi_{0}(t) = 1$ and
\begin{equation}\label{eq:phi-N}
  \phi_{N +1}(t) = \frac{a t^2 \phi_{N}(t)}{(1 - t)(1 - b t) - a t^2\phi_{N}(t)} \, ,
\end{equation}
where ${a > 0}$ and ${b \ge 1}$ are constants.
\end{definition}

\begin{correctness41}
We prove the correctness of the definition by induction.
For $N = 0$ it is obvious.
Assume that for some ${N \ge 0}$ the function $\phi_N$ is well-defined and nondecreasing on ${[0,R]}$ and ${\phi_N(R) = 1}$.
We shall prove the same for $\phi_{N + 1}$.
It follows from the induction hypothesis that
\begin{equation} \label{eq:denominator-of-phi-N-positive}
(1 - t) (1 - b t) - a t^2 \phi_N(t) \ge  (1 - t) (1 - b t) - a t^2 > 0 \quad\text{for all } t \in [0,R]
\end{equation}
which means that the function ${\phi_{N + 1}}$ is well-defined on ${[0, R]}$.
From \eqref{eq:phi-N} and the induction hypothesis, we deduce that ${\phi_{N + 1}}$ is nondecreasing on [0,R].
From \eqref{eq:phi-N} and ${\phi_N(R) = 1}$, we obtain
\[
  \phi_{N + 1}(R)  =  \displaystyle \frac{a R^2 \phi_N(R)}{(1 - R)(1 - b R) - a R^2 \phi_N(R)}
   =  \displaystyle \frac{a R^2}{(1 - R)(1- b R) - a R^2} = \phi(R) = 1.
\]
This completes the induction and the proof of the correctness of Definition~\ref{def:phi-N}.
\qed
\end{correctness41}

\begin{definition} \label{def:varphi-N}
For any integer ${N \ge 0}$, we define the function ${\varphi_N \colon [0,R] \to [0,R]}$ as follows
\begin{equation} \label{eq:varphi-N}
  \varphi_{N}(t) = t \phi_{N}(t),
\end{equation}
where the function $\phi_{N}$ is defined by Definition~\ref{def:phi-N}.
\end{definition}

In the next lemma, we present some properties of the functions $\phi_N$ and $\varphi_N$.

\begin{lemma} \label{lem:phi-N-properties}
Let ${N \ge 0}$. Then:
\begin{enumerate}
 \item $\phi_N$ is a quasi-homogeneous function of degree ${2 N}$ on ${[0,R]}$;
 \item ${\phi_{N+1}(t) \le \phi(t) \, \phi_{N}(t)}$ for every ${t \in [0,R]}$;
 \item ${\phi_{N+1}(t) \le \phi_{N}(t)}$ for every ${t \in [0,R]}$;
 \item ${\phi_N(t) \le \phi(t)^N}$ for every ${t \in [0,R]}$;
 \item $\varphi_N$ is a gauge function of order ${2 N + 1}$ on ${[0,R]}$.
\end{enumerate}
\end{lemma}

\begin{proof}
Claim (i) can easily be proved by induction.
From \eqref{eq:phi-N} and \eqref{eq:denominator-of-phi-N-positive}, we get
\[
  \phi_{N + 1} (t) = \frac{a t^2 \phi_{N}(t)}{(1- t)(1 - b t) - a t^2 \phi_{N}(t)}
   \le  \frac{a t^2 \phi_{N}(t)}{(1- t)(1 - b t) - a t^2} = \phi(t) \, \phi_{N}(t)
\]
which proves (ii). Claim (iii) is a trivial consequence from (ii).
Claim (iv) follows from (ii) by induction.
Claim (v) follows from (i) and the definition of $\varphi_N$.
\end{proof}

\begin{lemma} \label{lem:Ehrlich-high-order-map}
Let ${f \in \Kset[z]}$ be a polynomial of degree ${n \ge 2}$, ${\xi \in \Kset^n}$ be a root-vector of $f$ and ${N \ge 0}$.
Suppose ${x \in D_N}$ is a vector such that ${f(x_i) \ne 0}$ for some ${i \in I_n}$.

\emph{(i)} If ${x \, \# \,T^{(N)}(x)}$, then
\begin{equation} \label{eq:equation-A}
\frac{f'(x_i)}{f(x_i)} - \sum_{j \ne i} \frac{1}{x_i - T_j^{(N)}(x)} = \frac{1 - \sigma_i}{x_i -\xi_i},
\end{equation}
where ${\sigma_i \in \Kset}$ is defined by
\begin{equation} \label{eq:sigma-definition}
\sigma_i = (x_i - \xi_i) \sum_{j \ne i}
\frac{ T_j ^{(N)}(x) - \xi_j}{(x_i - \xi_j)(x_i - T_j ^{(N)}(x))} \,.
\end{equation}

 \emph{(ii)} If ${x \in D_{N + 1}}$, then
 \begin{equation} \label{eq:Ehrlich-high-order-map}
 T^{(N + 1)}_i(x) - \xi_i = - \frac{ \sigma_i}{1 - \sigma_i} (x_i - \xi_i).
\end{equation}
\end{lemma}

\begin{proof}
(i)
Taking into account that $\xi$ is a root-vector of $f$, we get
\begin{eqnarray*}
\frac{f'(x_i)}{f(x_i)} - \sum_{j \ne i} \frac{1}{x_i - T_j^{(N)}(x)} & = & \sum_ {j = 1} ^n \frac{1}{x_i - \xi_j} - \sum_{j \ne i} \frac{1}{x_i - T_j^{(N)}(x)} \\
& = & \frac{1}{x_i - \xi_i} + \sum_ {j \ne i} \left( \frac{1}{x_i - \xi_j} - \frac{1}{x_i - T_j ^{(N)}(x)} \right) \\
& = & \frac{1}{x_i - \xi_i} - \sum_{j \ne i}\frac{T_j^ {(N)} (x)- \xi_j}{(x_i - \xi_j)(x_i - T_j ^{(N)}(x))} =  \frac{1 - \sigma_i}{x_i - \xi_i}
\end{eqnarray*}
which proves \eqref{eq:equation-A}.

(ii)
It follows from ${x \in D_{N + 1}}$ that
\begin{equation} \label{eq:correction-function-denominator}
f'(x_i) - f(x_i) \sum_{j \ne i} \frac{1}{x_i - T_j^{(N)}(x)}  \ne 0.
\end{equation}
Then from \eqref{eq:Ehrlich-high-order-iteration-function} and \eqref{eq:equation-A}, we obtain
\begin{eqnarray*}
T^{(N + 1)}_i(x) - \xi_i & = & x_i - \xi_i - \displaystyle\left(\frac{f'(x_i)}{f(x_i)} - \sum_{j \ne i} \frac{1}{x_i - T_j^{(N)}(x)} \right)^{-1} \\
& = &  x_i - \xi_i - \displaystyle \frac{x_i - \xi_i}{1 - \sigma_i} = - \frac{ \sigma_i}{1 - \sigma_i}(x_i - \xi_i)
\end{eqnarray*}
which completes the proof.
\end{proof}

\begin{lemma} \label{lem:Ehrlich-iterated-contraction}
Let ${f \in \Kset[z]}$ be a polynomial of degree ${n \ge 2}$ which has  only simple zeros in $\Kset$, ${\xi \in \Kset^n}$ be a root-vector of $f$, ${N \ge 0}$ and ${1 \le p \le \infty}$. Suppose ${x \in \Kset^n}$ is a vector satisfying the following condition
\begin{equation} \label{eq:Ehrlich-FIN-condition}
E(x) = \left \| \frac{x - \xi}{d(\xi)} \right \|_p < R = \frac{2}{b + 1 + \sqrt{(b - 1)^ 2 + 8 a}}\, ,
\end{equation}
where the function ${E \colon \Kset^n \to \Rset_+}$ is defined by \eqref{eq:FIC-first}, ${a = (n - 1) ^{1/q}}$ and ${b = 2^ {1/q}}$.
Then
\begin{equation} \label{eq:Ehrlich-first-local-relations}
  x \in D_N, \, \|T^{(N)} (x) - \xi\| \preceq \phi_N(E(x)) \|x - \xi\| \quad\text{and}\quad  E(T^{(N)}(x)) \le \varphi_N(E(x)).
\end{equation}
\end{lemma}

\begin{proof}
We shall prove statements by induction on $N$.
If ${N = 0}$, then \eqref{eq:Ehrlich-first-local-relations} holds trivially.
Assume that \eqref{eq:Ehrlich-first-local-relations} holds for some ${N \ge 0}$.

First, we show that ${x \in D_{N + 1}}$, i.e.
${x \, \# \, T^{(N)}(x)}$ and \eqref{eq:correction-function-denominator} holds for every ${i \in I_n}$.
It follows from the first inequality in \eqref{eq:Ehrlich-first-local-relations}
that the inequality \eqref{eq:inequality-vectors-cone-norm} is satisfied with ${u = x}$ and ${v = T^{(N)}(x)}$.
Then by Lemma~\ref{lem:u-v-inequalities-1} and \eqref{eq:Ehrlich-FIN-condition}, we obtain
\begin{equation} \label{eq:x-TN-inequality}
|x_i - T_j^{(N)}(x)| \ge \left( 1 - b \left\| \frac{x - \xi}{d(\xi)} \right\|_p \right)|\xi_i - \xi_j|
 \ge (1 -  b E(x)) \, d_j(\xi) > 0
\end{equation}
for every ${j \ne i}$.
Consequently, ${x \, \# \, T^{(N)}(x)}$.
It remains to prove \eqref{eq:correction-function-denominator} for every ${i \in I_n}$.
Let ${i \in I_n}$ be fixed.
We shall consider only the non-trivial case ${f(x_i) \ne 0}$.
In this case \eqref{eq:correction-function-denominator} is equivalent to
\begin{equation} \label{eq:correction-function-denominator-second-case}
\frac{f'(x_i)}{f(x_i)} - \sum_{j \ne i} \frac{1}{x_i - T_j^{(N)}(x)}  \ne 0.
\end{equation}
We define $\sigma_i$ by \eqref{eq:sigma-definition}. It follows from Lemma~\ref{lem:Ehrlich-high-order-map}(i) that \eqref{eq:correction-function-denominator-second-case}
is equivalent to ${\sigma_i \ne 1}$.
By Lemma~\ref{lem:u-v-inequalities-basic} with ${u = x}$ and ${v = \xi}$ and \eqref{eq:Ehrlich-FIN-condition}, we get
\begin{equation}\label{eq:x-xi-inequality}
		|x_i - \xi_j| \ge \left( 1 -  \left\| \frac{x - \xi}{d(\xi)} \right\|_p \right) \, |\xi_i - \xi_j|
		\ge (1 - E(x))\, d_i(\xi) > 0
\end{equation}
for every ${j \ne i}$.
From the triangle inequality in $\Kset$, \eqref{eq:x-TN-inequality}, \eqref{eq:x-xi-inequality}, induction hypothesis
and H\"older's inequality, we get
\begin{eqnarray} \label{eq:sigma-inequality}
|\sigma_i| & \le & |x_i - \xi_i| \sum_{j \ne i}\frac{|T_j ^{(N)}(x) - \xi_j|}{|x_i - \xi_j| \;
|x_i - T_j ^{(N)}(x)|} \nonumber \\
& \le & \frac{1}{(1 - E(x)) (1 - b E(x))}\frac{|x_i - \xi_i|}{d_i(\xi)}
\sum_{j \ne i} \frac{|T_j ^{(N)} (x) - \xi_j|}{d_j(\xi)} \nonumber \\
& \le & \frac{a E(x) \varphi_N (E(x))}{(1 - E(x)) (1 - b E(x))}  = \frac{a E(x)^2 \phi_N (E(x))}{(1 - E(x)) (1 - b E(x))} \, .
\end{eqnarray}
From this, ${\phi_N(E(x)) \le 1}$ and \eqref{eq:Ehrlich-FIN-condition}, we obtain
\[
|\sigma_i| \le \frac{a E(x)^2}{(1 - E(x))(1 - b E(x))} < 1
\]
which yields ${\sigma_i \ne 1}$ and so \eqref{eq:correction-function-denominator} holds.
Hence, $ x \in D_{N + 1}$.

Second, we show that the inequalities in \eqref{eq:Ehrlich-first-local-relations} hold for ${N + 1}$.
The first inequality for ${N + 1}$ is equivalent to
\begin{equation} \label{eq:Ehrlich-contraction-map-inequality-i-coordinate}
 |T^{(N + 1)}_i(x) - \xi_i| \le  \phi_{N + 1}(E(x)) |x_i - \xi_i| \quad\text{for all } i \in I_n \, .
\end{equation}
Let ${i \in I_n}$ be fixed. If ${x_i = \xi_i}$, then  ${T^{(N + 1)}_i(x) = \xi_i}$ and so \eqref{eq:Ehrlich-contraction-map-inequality-i-coordinate} becomes an equality.
Suppose ${x_i \ne \xi_i}$.
By Lemma~\ref{lem:Ehrlich-high-order-map}(ii), the triangle inequality in $\Kset$ and the estimate \eqref{eq:sigma-inequality}, we get
\begin{eqnarray*}
 |T^{(N + 1)}_i(x) - \xi_i| & = & \frac{|\sigma_i|}{|1 - \sigma_i|} | x_i - \xi_i|  \le \frac{|\sigma_i|}{1 - | \sigma_i|} | x_i - \xi_i| \\
 & \le & \displaystyle  \frac{a E(x)^2 \phi_N(E(x))}{(1 - E(x))(1 - b E(x)) - a E(x)^2 \phi_N(E(x))} | x_i - \xi_i|\\
& = & \phi_{N + 1}(E(x)) |x_i - \xi_i|
\end{eqnarray*}
which proves \eqref{eq:Ehrlich-contraction-map-inequality-i-coordinate}.
Dividing both sides of the inequality \eqref{eq:Ehrlich-contraction-map-inequality-i-coordinate} by ${d_i(\xi)}$ and taking the $p$-norm, we obtain
\[
 E(T^{(N + 1)}(x)) \le \varphi_{N + 1}(E(x))
\]
which proves that the second inequality in \eqref{eq:Ehrlich-first-local-relations} holds for ${N + 1}$.
This completes the induction and the proof of the lemma.
\end{proof}

Now we are ready to state the main result of this section.
In the case ${N = 1}$ this result coincides with Theorem~\ref{thm:Proinov-Ehrlich-local-first}.

\begin{theorem}\label{thm:first-local-Ehrlich}
Let ${f \in \Kset[z]}$ be a polynomial of degree ${n \ge 2}$ which has  only simple zeros in $\Kset$,
${\xi \in \Kset^n}$ be a root-vector of $f$, ${N \ge 1}$ and ${1 \le p \le \infty}$.
Suppose ${x^{(0)} \in \Kset^n}$ is an initial guess satisfying
\begin{equation} \label{eq:first-local-Ehrlich-initial-condition}
E(x^{(0)}) = \left \| \frac{x^{(0)} - \xi}{d(\xi)} \right \|_p < R = \frac{2}{b + 1 + \sqrt{(b - 1)^2 + 8a}} \, ,
\end{equation}
where the function ${E \colon \Kset^n \to \Rset_+}$ is defined by \eqref{eq:FIC-first}, ${a = (n - 1) ^{1/q}}$ and ${b = 2^ {1/q}}$.
Then the Ehrlich-type iteration \eqref{eq:Ehrlich-high-order} is well-defined and converges to $\xi$ with error estimates
\begin{equation}\label{eq:first-local-Ehrlich-error-estimate}
\|x^{(k+1)} - \xi\| \preceq \lambda^{(2 N + 1)^k} \, \|x^{(k)} - \xi\|
\quad\text{and}\quad
  \|x^{(k)} - \xi\| \preceq \lambda^{((2 N + 1)^k - 1) / (2N)} \|x^{(0)} - \xi\|
\end{equation}
for all ${k \ge 0}$, where ${\lambda = \phi_N(E(x^{(0)}))}$ and the function $\phi_N$ is defined by Definition~\ref{def:phi-N}.
\end{theorem}

\begin{proof}
We apply Corollary~\ref{cor:general-convergence-theorem} to the iteration function
${T^{(N)} \colon D_N \subset \Kset^n \to \Kset^n}$ defined by Definition~\ref{def:iteration-functions} and to the function
${E \colon \Kset^n \to \Rset_+}$ defined by \eqref{eq:FIC-first}.
Let ${J = [0, R)}$.
It follows from the Lemma~\ref{lem:Ehrlich-iterated-contraction}, Lemma~\ref{lem:phi-N-properties}(v) and 
Lemma~\ref{lem:gauge-function-condition} that
$E$ is a function of initial conditions of $T^{(N)}$ with a strict gauge function $\varphi_N$ of order ${r = 2 N + 1}$ on $J$.
Since $\xi$ is a root-vector of $f$, then ${E(\xi) = 0 \in J}$.
It follows from Lemma~\ref{lem:Ehrlich-iterated-contraction}, that ${T^{(N)}}$ is an iterated contraction at a point $\xi$ with respect to
$E$ and with control function $\phi_N$.
The fact that ${x^{(0)}}$ is an initial point of ${T^{(N)}}$ follows from
Lemma~\ref{lem:Ehrlich-iterated-contraction} and Theorem~\ref{thm:initial-point-test}.
Hence, all the assumptions of Corollary~\ref{cor:general-convergence-theorem} are satisfied, and the statement of
Theorem~\ref{thm:first-local-Ehrlich} follows from it.
\end{proof}

\begin{corollary}  \label{cor:local-Ehrlich}
Let ${f \in \Kset[z]}$ be a polynomial of degree ${n \ge 2}$ which has  only simple zeros in $\Kset$,
${\xi \in \Kset^n}$ be a root-vector of $f$, ${N \ge 1}$ and ${1 \le p \le \infty}$.
Suppose ${x^{(0)} \in \Kset^n}$ is an initial guess satisfying
\eqref{eq:first-local-Ehrlich-initial-condition}.
Then the Ehrlich-type iteration \eqref{eq:Ehrlich-high-order} is well-defined and converges to $\xi$ with error estimates
\begin{equation} \label{eq:local-Ehrlich-cor-error-estimate}
\|x^{(k+1)} - \xi\| \preceq \lambda^{N (2 N + 1)^k} \, \|x^{(k)} - \xi\|
\quad\text{and}\quad
  \|x^{(k)} - \xi\| \preceq \lambda^{((2 N + 1)^k - 1) / 2} \|x^{(0)} - \xi\|
\end{equation}
for all ${k \ge 0}$, where ${\lambda = \phi(E(x^{(0)}))}$ and $\phi$ is a real function defined by \eqref{eq:phi-Proinov}.
\end{corollary}

\begin{proof}
It follows from Theorem~\ref{thm:first-local-Ehrlich} and Lemma~\ref{lem:phi-N-properties}(iv).
\end{proof}

Let ${0 < h < 1}$ be a given number. Solving the equation ${\phi(t) = h^2}$ in the interval ${(0,R)}$, we can reformulate
Corollary~\ref{cor:local-Ehrlich} in the following equivalent form.

\begin{corollary} \label{cor:first-local-Ehrlich-var}
Let ${f \in \Kset[z]}$ be a polynomial of degree $n \ge 2$ which has $n$ simple zeros in $\Kset$, ${\xi \in \Kset^n}$ be a root-vector of
$f$, ${N \ge 1}$, ${1 \le p \le \infty}$ and ${0 < h < 1}$.
Suppose ${x^{(0)} \in \Kset^n}$ is an initial guess which satisfies
\begin{equation} \label{eq:first-local-Ehrlich-var-initial-condition}
E(x^{(0)}) = \left \| \frac{x^{(0)} - \xi}{d(\xi)} \right \|_p  < R_h = \frac{2}{b + 1 + \sqrt{ (b - 1)^2 + 4 a (1 + 1/h^2)}}\,,
\end{equation}
where ${a = (n - 1) ^{1/q}}$ and ${b = 2^ {1/q}}$.
Then the Ehrlich-type method \eqref{eq:Ehrlich-high-order} is well-defined and converges to $\xi$ with error estimates
\begin{equation} \label{eq:first-local-Ehrlich-var-error-estimate}
\|x^{(k+1)} - \xi\| \preceq h^{2 N (2 N + 1)^k} \, \|x^{(k)} - \xi\|
\quad\text{and}\quad
\|x^{(k)} - \xi\| \preceq h^{(2 N + 1)^k - 1} \|x^{(0)} - \xi\|
\end{equation}
for all ${k \ge 0}$.
\end{corollary}

\begin{remark}
Corollary~\ref{cor:first-local-Ehrlich-var} is an improvement of the
result of Kjurkchiev and Andreev \cite{KA87} (see Theorem~\ref{thm:Kjurkchiev-Andreev} above).
Suppose that a vector ${x^{(0)} \in \Kset^n}$ satisfies \eqref{eq:Kjurkchiev-Andreev-initial-condition}.
It is easy to show that condition \eqref{eq:Kjurkchiev-Andreev-h-c} is equivalent to the following one:
\[
0 < c < \min \left\{ \frac{\delta}{2(1 + (2n - 1)h)},\frac{2 \delta}{3 + 2 h + \sqrt{4(3 n - 2) h^2 + 4 h + 4 n + 1}}  \right\} .
\]
From this, the initial condition \eqref{eq:Kjurkchiev-Andreev-initial-condition} and ${0 < h < 1}$, we obtain
\begin{eqnarray*}
\left \| \frac{x^{(0)} - \xi}{d(\xi)} \right \|_\infty & \le &  \frac{\|x^{(0)} - \xi \|_\infty}{\delta} \le \frac{c h}{\delta} \\
& \le & \frac{2 h}{3 + 2 h + \sqrt{4(3 n - 2) h^2 + 4 h + 4 n + 1}} \\
& \le & \frac{2}{3 + \sqrt{4(3 n - 2) + (4 n + 1)/h^2}} \\
& \le & \frac{2}{3 + \sqrt{4n - 3 + 4(n - 1)/h^2}} \, .
  \end{eqnarray*}
Therefore, $x^{(0)}$ satisfies \eqref{eq:first-local-Ehrlich-var-initial-condition} with $p = \infty$.
Then it follows from Corollary~\ref{cor:first-local-Ehrlich-var} that
the Ehrlich-type method \eqref{eq:Ehrlich-high-order} is well-defined and converges to $\xi$ with error estimates
\eqref{eq:first-local-Ehrlich-var-error-estimate}.
From the second estimate in \eqref{eq:first-local-Ehrlich-var-error-estimate} and \eqref{eq:Kjurkchiev-Andreev-initial-condition},
we get the estimate \eqref{eq:Kjurkchiev-Andreev-error-estimate} which completes the proof.
\end{remark}


\section{Local convergence theorem of the second type}
\label{sec:Local-convergence-theorem-of-the-second-type}

Let ${f \in \Kset[z]}$ be a polynomial of degree $n \ge 2$.
We study the convergence of the Ehrlich-type method \eqref{eq:Ehrlich-high-order} with respect to the  function of initial conditions
${E \colon \mathcal{D} \to \Rset_+}$ defined by
\begin{equation} \label{eq:FIC-second}
E(x) = \left \|\frac{x - \xi}{d(x)} \right \|_p \qquad (1 \le p \le \infty).
\end{equation}

In the previous section, we introduce the functions $\phi_N$, $\varphi_N$ and the real number $R$ with two parameters
${a > 0}$ and ${b \ge 1}$. In this section, we consider a special case of $\phi_N$, $\varphi_N$ and $R$ when ${b = 2}$.
In other words, now we define $R$ by
\begin{equation} \label{eq:R-second}
R =\frac{2}{3 + \sqrt{1 + 8 a}} \, .
\end{equation}
Furthermore, we define the functions $\phi_N$ and $\varphi_N$ by Definitions \ref{def:phi-N} and \ref{def:varphi-N}, respectively,
but with
\begin{equation}\label{eq:phi-N-second}
  \phi_{N +1}(t) = \frac{a t^2 \phi_{N}(t)}{(1 - t)(1 - 2 t) - a t^2\phi_{N}(t)}
\end{equation}
instead of \eqref{eq:phi-N}, where ${a > 0}$ is a constant.

\begin{definition} \label{def:beta-psi-N}
For a given integer $N \ge 1$, we define the increasing function ${\beta_N \colon [0,R] \to [0,1)}$ by
\begin{equation}\label{eq:beta-N}
   \beta_{N} (t) = \frac{a t^2 \phi_{N-1}(t)}{1 - t - a t^2 \phi_{N-1}(t)}
\end{equation}
and we define the decreasing function ${\psi_N \colon [0,R] \to (0,1]}$ as follows
\begin{equation} \label{eq:psi-N}
   \psi_{N} (t) = 1 - 2 t (1 + \beta_{N}(t)) = \frac{(1 - t)(1 - 2 t) - a t^2 \phi_{N-1}(t)}{1 - t - a t^2 \phi_{N-1}(t)} \, .
\end{equation}
\end{definition}

\begin{correctness51}
The functions ${\beta_N}$ and ${\psi_N}$ are well-defined on ${[0, R]}$ since
\begin{equation} \label{eq:denominator-of-beta-N-positive}
1 - t - a t^2 \phi_{N-1}(t) \ge  1 - t - a t^2 > 0 \quad\text{for all } t \in [0,R].
\end{equation}
The monotonicity of ${\beta_N}$ and ${\psi_N}$ is obvious. It remains to prove that ${\beta_N(R) < 1}$ and ${\psi_N(R) > 0}$.
Since ${\phi_{N}(R) = 1}$, we obtain
\[
\beta_N(R) =\frac{a R^2}{1 - R - a R^2} < 1 \quad\text{and}\quad
\psi_{N}(R) = \frac{(1 - R)(1 - 2 R) - a R^2}{1 - R - a R^2} > 0
\]
which completes the proof of the correctness of Definition~\ref{def:beta-psi-N}
\qed
\end{correctness51}

\begin{lemma} \label{lem:beta-psi-N-properties}
Let ${N \ge 1}$. Then:
\begin{enumerate}
  \item $\beta_N$ is a quasi-homogeneous of degree ${2 N}$ on ${[0,R]}$;
  \item ${\beta_N(t) = \phi_N(t) \psi_N(t)}$ for every ${t \in [0,R]}$;
  \item ${\beta_{N+1}(t) \le \beta_N(t)}$ for every ${t \in [0,R]}$;
  \item ${\psi_{N + 1}(t) \ge \psi_N(t)}$ for every ${t \in [0,R]}$.
\end{enumerate}
\end{lemma}

\begin{proof}
The function ${\beta_N}$ can be presented in the form ${\beta_N(t) = t^2 \phi_{N-1}(t) \Phi(t)}$, where 
${\Phi(t) = a / (1 - t -a t^2 \phi_{N-1}(t))}$. Therefore, ${\beta_N}$ is quasi-homogeneous of degree ${2 N}$ on ${[0, R]}$ since it is a product of three quasi-homogeneous functions on ${[0, R]}$ of degree $2$, ${2N-2}$ and $0$.  
From the definitions of the functions $\phi_N$, $\psi_N$ and $\beta_N$, we get
\[
  \phi_N(t) \psi_N(t) = \frac{a t^2 \phi_{N - 1}(t)}{(1 - t)(1 - 2 t) - a t^2\phi_{N - 1}(t)} \,
	\frac{(1 - t)(1 - 2 t) - a t^2 \phi_{N - 1}(t)}{1 - t - a t^2 \phi_{N - 1}(t)} = \beta_{N}(t).
\]
Claim (iii) follows from Lemma~\ref{lem:phi-N-properties}(iii) and \eqref{eq:beta-N}. 
Claim (iv) follows from (iii) and \eqref{eq:psi-N}.
\end{proof}

\begin{lemma} \label{lem:Ehrlich-iterated-contraction-second-FIC}
Let ${f \in \Kset[z]}$ be a polynomial of degree ${n \ge 2}$ which has  only simple zeros in $\Kset$,
${\xi \in \Kset^n}$  a root-vector of $f$, ${N \ge 1}$ and ${1 \le p \le \infty}$. 
Suppose ${x \in \Kset^n}$ is a vector with distinct components such that
\begin{equation} \label{eq:Ehrlich-second-FIN-condition}
E(x) = \left \| \frac{x - \xi}{d(x)} \right \|_p  \le R = \frac{2}{3 + \sqrt{1 + 8a}}\, ,
\end{equation}
where the function ${E \colon \mathcal{D} \to \Rset_+}$ is defined by \eqref{eq:FIC-second} and ${a = (n - 1) ^{1/q}}$.
Then $f$ has only simple zeros in $\Kset$,
 \begin{equation} \label{eq:Ehrlich-second-local-relations}
  x \in D_N, \, \|T^{(N)} (x) - \xi\| \preceq \beta_N(E(x)) \|x - \xi\| \quad\text{and}\quad  E(T^{(N)}(x)) \le \varphi_N(E(x)).
\end{equation}
Besides, the vector $T^{(N)}(x)$ has pairwise distinct components. 
\end{lemma}

\begin{proof}
It follows from \eqref{eq:Ehrlich-second-FIN-condition} and $R < 1/2$ that ${E(x) < 1/2}$. 
Then it follows from Lemma~\ref{lem:u-v-inequality-distinct-components} that the vector $\xi$ has distinct components, 
which means that $f$ has only simple zeros in $\Kset$.
We divide the proof into two steps.

\begin{step}
In this step, we prove ${x \in D_{N}}$ and the first inequality in
\eqref{eq:Ehrlich-second-local-relations} by induction on $N$.
If ${N = 1}$, the proof of the claims can be found in \cite{Pro15c}.
Assume that ${x \in D_{N}}$ and the first inequality in
\eqref{eq:Ehrlich-second-local-relations} hold for some ${N \ge 1}$.

First we show that ${x \in D_{N + 1}}$ i.e.
${x \, \# \, T^{(N)}(x)}$ and \eqref{eq:correction-function-denominator} holds for every ${i \in I_n}$.
It follows from the first inequality in \eqref{eq:Ehrlich-second-local-relations} that \eqref{eq:inequality-vectors-cone-norm-h-condition} holds with ${v = x}$, ${u = T^{(N)}(x)}$ and ${\alpha = 1}$.
Therefore by Lemma~\ref{lem:u-v-inequalities-2}, \eqref{eq:Ehrlich-second-FIN-condition} and ${R < 1/2}$, we obtain
\begin{equation} \label{eq:x-TN-inequality-second}
     |x_i - T_j^{(N)}(x)| > \left( 1 - 2 \left\| \frac{x - \xi}{d(x)} \right\|_p \right) |x_i - x_j| \ge (1 -  2 E(x)) \, d_j(x) > 0
\end{equation}
for every ${j \ne i}$.
Consequently, ${x \, \# \, T^{(N)}(x)}$.
It remains to prove \eqref{eq:correction-function-denominator} for every ${i \in I_n}$.
Let ${i \in I_n}$ be fixed.
We shall consider only the non-trivial case ${f(x_i) \ne 0}$.
In this case \eqref{eq:correction-function-denominator} is equivalent to \eqref{eq:correction-function-denominator-second-case}.
On the other hand, it follows from Lemma~\ref{lem:Ehrlich-high-order-map}(i) that \eqref{eq:correction-function-denominator-second-case}
is equivalent to ${\sigma_i \ne 1}$, where ${\sigma_i}$ is defined by \eqref{eq:sigma-definition}.
By Lemma~\ref{lem:u-v-inequalities-basic} with ${u = \xi}$ and ${v = x}$ and \eqref{eq:Ehrlich-second-FIN-condition}, we get
\begin{equation} \label{eq:x-xi-inequality-2}
    |x_i - \xi_j| \ge \left( 1 -  \left\| \frac{x - \xi}{d(x)} \right\|_p \right) |x_i - x_j|
 = (1 -  E(x)) |x_i - x_j| \ge (1 - E(x)) \, d_i(x) > 0,
\end{equation}
for every ${j \ne i}$.
Hence, we obtain ${x \, \# \, \xi}$.
From induction hypothesis, we get
\begin{equation} \label{eq:Ehrlich-contraction-map-inequality-j-coordinate-second}
     |T^{(N)}_i (x) - \xi_i | \le \beta_N(E(x)) |x_i - \xi_i |.
\end{equation}
Combining the triangle inequality in $\Kset$, \eqref{eq:x-xi-inequality-2}, \eqref{eq:x-TN-inequality-second} and \eqref{eq:Ehrlich-contraction-map-inequality-j-coordinate-second}, we obtain
\begin{eqnarray*}
|\sigma_i|& \le &|x_i - \xi_i| \sum_{j \ne i}\frac{|T_j ^{(N)} (x) - \xi_j|}{|x_i - \xi_j| \,
| x_i - T_j ^{(N)} (x) |} \\
& \le & \frac{1}{(1 - E(x)) (1 - 2 E(x))} \frac{|x_i - \xi_i|}{d_i(x)}
\sum_{j \ne i} \frac {|T_j ^{(N)} (x) - \xi_j|}{d_j(x)} \\
& \le & \frac{\beta_N(E(x))}{(1 - E(x)) (1 - 2 E(x))} \frac{|x_i - \xi_i|}{d_i(x)}
\sum_{j \ne i} \frac {|x_j - \xi_j|}{d_j(x)} \, .
\end{eqnarray*}
which, using H\"older's inequality, yields
\begin{equation} \label{eq:sigma-inequality-second}
|\sigma_i| \le \frac{a E(x)^2 \beta_N(E(x))}{(1 - E(x)) (1 - 2 E(x))} \, .
\end{equation}
From  this and \eqref{eq:Ehrlich-second-FIN-condition}, we deduce
\[
 |\sigma_i| \le \frac{a E(x)^2}{(1 - E(x)) (1 - 2 E(x))} < 1,
\]
which yields ${\sigma_i \ne 1}$, and so \eqref{eq:correction-function-denominator-second-case} holds.
Thus we prove that $ x \in D_{N + 1}$.
 
Now we have to prove that the first inequality in \eqref{eq:Ehrlich-second-local-relations} is satisfied for ${N + 1}$, 
which is equivalent to
\begin{equation} \label{eq:Ehrlich-contraction-map-inequality-i-coordinate-second}
 |T^{(N + 1)}_i(x) - \xi_i| \le  \beta_{N + 1}(E(x)) |x_i - \xi_i|  \quad\text{ for all }\, i \in I_n.
\end{equation}
Let ${i \in I_n}$ be fixed. If ${x_i = \xi_i}$, then  ${T^{(N + 1)}_i(x) = \xi_i}$  and the inequality 
\eqref{eq:Ehrlich-contraction-map-inequality-i-coordinate-second} becomes an equality.
Suppose ${x_i \ne \xi_i}$.
It follows from Lemma~\ref{lem:Ehrlich-high-order-map}(ii), the triangle inequality in $\Kset$ and the estimate 
\eqref{eq:sigma-inequality-second} that
\begin{eqnarray*}
|T^{(N + 1)}_i(x) - \xi_i| & = & \frac{|\sigma_i|}{|1 -\sigma_i|} |x_i - \xi_i| \le \frac{|\sigma_i|}{1 - |\sigma_i|} |x_i - \xi_i| \\
& \le & \frac{ a E(x)^2 \beta_N(E(x))}{(1 - E(x))(1 - 2 E(x)) - a E(x)^2 \beta_N(E(x))} |x_i - \xi_i|.
\end{eqnarray*}
From this inequality, Lemma~\ref{lem:beta-psi-N-properties}(ii), ${\psi_N(t) \le 1}$, \eqref{eq:phi-N-second} and 
Lemma~\ref{lem:beta-psi-N-properties}(iv), we obtain
\begin{eqnarray*}
|T^{(N + 1)}_i(x) - \xi_i| & \le & \frac{ a E(x)^2 \phi_N(E(x)) \psi_N(E(x))}{(1 - E(x))(1 - 2 E(x)) - a E(x)^2 \phi_N(E(x)) } | x_i - \xi_i | \\
& \le & \phi_{N + 1}(E(x)) \psi_N(E(x)) | x_i - \xi_i | \nonumber\\
& \le & \phi_{N + 1}(E(x)) \psi_{N + 1}(E(x)) | x_i - \xi_i |
 =  \beta_{N + 1}(E(x)) |x_i - \xi_i|
\end{eqnarray*}
which proves \eqref{eq:Ehrlich-contraction-map-inequality-i-coordinate-second}.
This completes the induction.
\end{step}

\begin{step}
In this step we prove the second inequality in \eqref{eq:Ehrlich-second-local-relations} 
and that $T^{(N)}(x)$ has distinct components.
First inequality in \eqref{eq:Ehrlich-second-local-relations} allow us to apply 
Lemma~\ref{lem:u-v-inequalities-3} with ${u = T^{(N)}(x)}$, ${v = x}$ and ${\alpha = \beta_N(E(x))}$.
By Lemma~\ref{lem:u-v-inequalities-3} and \eqref{eq:psi-N}, we deduce
\[
 |T_i^{(N)}(x) - T_j^{(N)}(x)| \ge (1 - 2^{1/q} E(x)(1 + \beta_N(E(x)))) |x_i - x_j| \ge \psi_N(E(x)) |x_i - x_j|.
\]
By taking the minimum over all ${j \in I_n}$ such that ${j \ne i}$, we obtain
\begin{equation} \label{eq:d-Tx-inequality-psi}
 d_i(T^{(N)}(x)) \ge \psi_N(E(x)) d_i(x) >0
\end{equation}
which implies that $T^{(N)}(x)$ has distinct components. 
It follows from \eqref{eq:Ehrlich-contraction-map-inequality-j-coordinate-second}, \eqref{eq:d-Tx-inequality-psi} and 
Lemma~\ref{lem:beta-psi-N-properties}(ii) that
\[
   \frac{|T_i^{(N)}(x) - \xi_i|}{ d_i(T^{(N)}(x)) }
   \le \frac{\beta_N(E(x))}{\psi_N(E(x))} \frac{|x_i - \xi_i|}{d_i(x)}
   = \phi_{N}(E(x)) \frac{|x_i - \xi_i|}{d_i(x)} \, .
\]
By taking the $p$-norm, we obtain
\[
   E(T^{(N)} (x)) \le \phi_{N}(E(x))\, E(x) = \varphi_N(E(x))
\]
\end{step}
which proves the second inequality in \eqref{eq:Ehrlich-second-local-relations}. This completes the proof.
\end{proof}

Now we are able to state the main result of this section.
In the case when ${N = 1}$ and ${p = \infty}$ this result reduces to Theorem~\ref{thm:Proinov-Ehrlich-local-second}.

\begin{theorem} \label{thm:second-local-Ehrlich}
Let ${f \in \Kset[z]}$ be a polynomial of degree ${n \ge 2}$ which splits over $\Kset$,
${\xi \in \Kset^n}$ be a root-vector of $f$, ${N \ge 1}$ and ${1 \le p \le \infty}$.
Suppose ${x^{(0)} \in \Kset^n}$ is an initial guess with distinct components such that
\begin{equation} \label{eq:second-local-Ehrlich-initial-condition}
E(x^{(0)}) = \left \| \frac{x^{(0)} - \xi}{d(x^{(0)})} \right \|_p \le R = \frac{2}{3 + \sqrt{1 + 8 a}},
\end{equation}
where the function $E$ is defined by \eqref{eq:FIC-second} and  ${a = (n - 1) ^{1/q}}$.
Then $f$ has only simple zeros in $\Kset$ and  the Ehrlich-type iteration \eqref{eq:Ehrlich-high-order} is well-defined and converges to 
$\xi$ with error estimates
\begin{equation}\label{eq:second-local-Ehrlich-error-estimate}
\|x^{(k+1)} - \xi\| \preceq \theta \lambda^{(2 N + 1)^k} \, \|x^{(k)} - \xi\|
\quad\text{and}\quad
  \|x^{(k)} - \xi\| \preceq \theta^ k \lambda^{((2 N + 1)^k - 1) /2N} \|x^{(0)} - \xi\|,
\end{equation}
for all ${k \ge 0}$, where ${\lambda = \phi_N(E(x^{(0)}))}$, ${\theta = \psi_N(E(x^{(0)}))}$.
Moreover, the method is convergent with order ${2N + 1}$ provided that ${E(x^{(0)}) < R}$.
\end{theorem}

\begin{proof}
We apply Theorem~\ref{thm:general-convergence-theorem} to the iteration function
${T^{(N)} \colon D_N \subset \Kset^n \to \Kset^n}$ together with the function ${E \colon D_N \to \Rset_+}$ defined by 
\eqref{eq:FIC-second}.

It follows from Lemma~\ref{lem:Ehrlich-iterated-contraction-second-FIC} and Lemma~\ref{lem:phi-N-properties}(v) 
that $E$ is a function of initial conditions of ${T^{(N)}}$ with gauge function $\varphi_N$ of order ${r = 2 N + 1}$ on the interval 
${J = [0, R]}$.

From Lemma~\ref{lem:Ehrlich-iterated-contraction-second-FIC}, we get that ${T^{(N)}}$ is an iterated contraction at $\xi$ 
with respect to $E$ and with control function ${\beta_N}$. 
Also, it is easy to see that the functions $\beta_N$, $\phi_N$, $\psi_N$ and $\varphi_N$ have the properties 
\eqref{eq:Proinov-beta-gauge-property}, \eqref{eq:Proinov-beta-property} and \eqref{eq:Proinov-phi-property}.  

It follows from Lemma~\ref{lem:Ehrlich-iterated-contraction-second-FIC} that ${x^{(0)} \in D_N}$. 
According to Theorem~\ref{thm:initial-point-test} to prove that ${x^{(0)}}$ is an initial point of ${T^{(N)}}$ 
it is sufficient to prove that
\begin{equation} \label{eq:initial-point-test}
	x \in D_N \,\, \text{ and } \,\, E(x) \in J \,\, \Rightarrow \,\, T^{(N)}(x) \in D_N \, .
\end{equation}
From ${x \in D_N}$, we have ${T^{(N)}(x) \in \Kset^n}$.
By Lemma~\ref{lem:Ehrlich-iterated-contraction-second-FIC}, ${T^{(N)}}$ has distinct components and 
${E(T^{(N)}(x)) \le \varphi_N(E(x))}$. 
The last inequality yields ${E(T^{(N)}(x)) \in J}$ since ${\varphi \colon J \to J}$ and ${E(x) \in J}$.
Thus we have both ${T^{(N)}(x) \in \mathcal{D}}$ and ${E(T^{(N)}(x)) \in J}$.
Applying Lemma~\ref{lem:Ehrlich-iterated-contraction-second-FIC} to the vector ${T^{(N)}(x)}$, we get ${T^{(N)}(x) \in D_N}$
which proves \eqref{eq:initial-point-test}.
Therefore, ${x^{(0)}}$ is an initial point of ${T^{(N)}}$.

Now the statement of Theorem~\ref{thm:second-local-Ehrlich} follows from  Theorem~\ref{thm:general-convergence-theorem}.
\end{proof}


\section{Semilocal convergence theorem}
\label{sec:semilocal-convergence-theorem}

In this section we establish semilocal convergence theorems for Ehrlich-type methods \eqref{eq:Ehrlich-high-order} for finding all zeros of a polynomial simultaneously.
We study the convergence of these methods with respect to the  function of initial conditions
${E \colon \mathcal{D} \to \Rset_+}$  defined by
\begin{equation} \label{eq:FIC-third}
E_f(x) =  \left \| \frac{W_f(x)}{d(x)} \right \|_p \qquad (1 \le p \le \infty).
\end{equation}

Recently Proinov \cite{Pro15b} has shown that there is a relationship between local and semilocal theorems for simultaneous root-finding methods. It turns out that from any local convergence theorem for a simultaneous method one can obtain as a consequence a semilocal theorem for the same method. In particular, from Theorem~\ref{thm:first-local-Ehrlich} we can obtain a semilocal convergence theorem for Ehrlich-type methods \eqref{eq:Ehrlich-high-order} under computationally verifiable initial conditions.
For this purpose we need the following result.

\begin{theorem}[Proinov \cite{Pro15b}] \label{prop:semilocal-theorem-second-kind}
Let ${f \in \Kset[z]}$ be a polynomial of degree ${n \ge 2}$.
Suppose ${x \in \Kset^n}$ is an initial guess with distinct components such that
\begin{equation}\label{eq:semilocal-theorem-initial-condition-second-kind}
   \left \| \frac{W_f(x)}{d(x)} \right \|_p \le \frac{R (1 - R)}{1 + (a - 1) R}
\end{equation}
for some ${1 \le p \le \infty}$ and ${0 < R \le 1/(1 + \sqrt{a})}$, where ${a = (n - 1) ^ {1/q}}$.
In the case ${n = 2}$ and ${p = \infty}$ we assume that inequality in
\eqref{eq:semilocal-theorem-initial-condition-second-kind} is strict.
Then $f$ has only simple zeros in $\Kset$ and there exists a root-vector ${\xi \in \Kset^n}$ of $f$ such that
\begin{equation} \label{eq:semilocal-error-estimate-initial-condition}
	\|x - \xi\| \preceq \alpha (E_f(x)) \, \|W(x)\| \quad\text{ and }\quad
	\left \| \frac{x - \xi}{d(x)} \right \|_p \le R,
\end{equation}	
where the real function $\alpha$ is defined by
\begin{equation} \label{eq:function-alpha-definition}
\alpha(t) = \frac{2}{1 - (a - 1) t + \sqrt{(1 - (a - 1) t)^2 - 4 t}} \, .
\end{equation}
If the inequality \eqref{eq:semilocal-theorem-initial-condition-second-kind} is strict,
then the second inequality in \eqref{eq:semilocal-error-estimate-initial-condition} is strict too.
\end{theorem}

Now, we are ready to state and prove the main result of this paper.

\begin{theorem} \label{thm:semilocal-theorem-Ehrlich-second}
Let ${f \in \Kset[z]}$ be a polynomial of degree ${n \ge 2}$, ${N \ge 1}$, ${1 \le p \le \infty}$.
Suppose ${x^{(0)} \in \Kset^n}$ is an initial guess with distinct components such that
\begin{equation}\label{eq:semilocal-theorem-Ehrlich-initial-condition-second}
E_f(x^{(0)}) = \left \| \frac{W(x^{(0)})}{d(x^{(0)})} \right \|_p < \frac{8}{(3 + \sqrt{1 + 8 a})^2} \, ,
\end{equation}
where the function ${E_f}$ is defined by \eqref{eq:FIC-third} and ${a = (n - 1) ^ {1/q}}$.
Then $f$ has only simple zeros in $\Kset$ and
the Ehrlich-type iteration \eqref{eq:Ehrlich-high-order} is well-defined and converges to a root-vector $\xi$ of $f$ with order of convergence ${2N + 1}$ and with a posteriori error estimate
\begin{equation} \label{eq:semilocal-error-estimate-Ehrlich}
  \|x^{(k)} - \xi\| \preceq \alpha (E_f(x^{(k)})) \, \|W_f(x^{(k)})\|,
\end{equation}
for all ${k \ge 0}$ such that ${E_f(x^{(k)}) < 8 / (3 + \sqrt{1 + 8 a})^2}$,
where the function $\alpha$ is defined by \eqref{eq:function-alpha-definition}.
\end{theorem}

\begin{proof}
Let us define $R$ by \eqref{eq:R-second}.
It is easy to calculate that ${R <  1 / (1 + \sqrt{a})}$ and
\[
 \frac{R (1 - R)}{1 + (a - 1)R}  =
\frac{2 (1 + \sqrt{1 + 8 a})}{(3 + \sqrt{1 + 8 a}) (1 + 2a + \sqrt{1 + 8 a})} = \frac{8}{(3 + \sqrt{1 + 8 a})^2} \, .
\]
Therefore, \eqref{eq:semilocal-theorem-Ehrlich-initial-condition-second} can be written in the form
\[
 \left \| \frac{W(x^{(0)})}{d(x^{(0)})} \right \|_p < \frac{R (1 - R)}{1 + (a - 1)R} \, .
\]
Then it follows from Theorem~\ref{prop:semilocal-theorem-second-kind} that $f$ has only simple zeros in $\Kset$ and
there exists a root-vector ${\xi \in \Kset^n}$ of $f$ such that
\[
\left \| \frac{x^{(0)} - \xi}{d(x^{(0)})} \right \|_p < R.
\]
Now Theorem~\ref{thm:second-local-Ehrlich} implies that the Ehrlich-type iteration \eqref{eq:Ehrlich-high-order} converges
to $\xi$ with order of convergence ${2N + 1}$.
It remains to prove the error estimate \eqref{eq:semilocal-error-estimate-Ehrlich}.
Suppose that for some ${k \ge 0}$,
\begin{equation}\label{eq:semilocal-theorem-Weierstrass-initial-condition-k}
   \left \| \frac{W_f(x^{(k)})}{d(x^{(k)})} \right \|_p < \frac{R (1 - R)}{1 + (a - 1) R} \, .
\end{equation}
Then it follows from Theorem~\ref{prop:semilocal-theorem-second-kind} that there exists a root-vector
${\eta\in \Kset^n}$ of $f$ such that
\begin{equation} \label{eq:semilocal-error-estimate-k}
  \|x^{(k)} - \eta\| \preceq \alpha (E_f(x^{(k)})) \, \|W_f(x^{(k)})\| \quad\text{and}\quad
	\left \| \frac{x^{(k)} - \eta}{d(x^{(k)})} \right \|_p < R.
\end{equation}
From the second inequality in \eqref{eq:semilocal-error-estimate-k} and Theorem~\ref{thm:second-local-Ehrlich}, we conclude that
the Ehrlich-type iteration \eqref{eq:Ehrlich-high-order} converges to $\eta$. By the uniqueness of the limit, we get ${\eta = \xi}$.
Therefore, the error estimate \eqref{eq:semilocal-error-estimate-Ehrlich} follows from the first inequality in
\eqref{eq:semilocal-error-estimate-k}.
This completes the proof.
\end{proof}

Setting ${p = \infty}$ in Theorem~\ref{thm:semilocal-theorem-Ehrlich-second}, we obtain the following result.

\begin{corollary}\label{cor:semilocal-theorem-Ehrlich-infinity}
Let ${f \in \Kset[z]}$ be a polynomial of degree ${n \ge 2}$ and ${N \ge 1}$. 
Suppose ${x^{(0)} \in \Kset}$ is an initial guess with distinct components such that
\begin{equation}\label{eq:semilocal-theorem-Ehrlich-initial-condition-infinity}
    \left \| \frac{W(x^{(0)})}{d(x^{(0)})} \right \|_\infty < \frac{8}{(3 + \sqrt{8n - 7})^2} .
\end{equation}
Then $f$ has only simple zeros in $\Kset$ and
the Ehrlich-type iteration \eqref{eq:Ehrlich-high-order} is well-defined and converges to a root-vector $\xi$ of $f$ with order of convergence ${2N + 1}$ and with error estimate \eqref{eq:semilocal-error-estimate-Ehrlich} for ${p = \infty}$.
\end{corollary}

Setting ${p = 1}$ in Theorem~\ref{thm:semilocal-theorem-Ehrlich-second} we obtain the following result.

\begin{corollary}\label{cor:semilocal-theorem-Ehrlich-1}
Let ${f \in \Kset[z]}$ be a polynomial of degree ${n \ge 2}$ and ${N \ge 1}$.
Suppose ${x^{(0)} \in \Kset^n}$ is an initial guess with distinct
components such that
\begin{equation}\label{eq:semilocal-theorem-Ehrlich-initial-condition-1}
    \left \| \frac{W(x^{(0)})}{d(x^{(0)})} \right \|_1 <  \frac{2}{9} \, .
\end{equation}
Then $f$ has only simple zeros in $\Kset$ and
the Ehrlich-type iteration \eqref{eq:Ehrlich-high-order} is well-defined and converges with order ${2N + 1}$ to a root-vector $\xi$ of $f$ with error estimate \eqref{eq:semilocal-error-estimate-Ehrlich} for ${p = 1}$.
\end{corollary}


\section{Numerical examples}
\label{sec:numerical-examples}

In this section, we present several numerical examples to show some applications of
Theorem~\ref{thm:semilocal-theorem-Ehrlich-second}.
Let ${f \in \Cset[z]}$ be a polynomial of degree ${n \ge 2}$ and let $x^{(0)} \in \Cset^n$ be an initial guess.
We show that Theorem~\ref{thm:semilocal-theorem-Ehrlich-second} can be used:
\begin{itemize}
	\item to prove numerically that $f$ has only simple zeros;
\item to prove numerically that  $N$th Ehrlich-type iteration \eqref{eq:Ehrlich-high-order}  starting from ${x^{(0)}}$  is well-defined and converges with order ${2N + 1}$ to a root-vector of $f$;
\item to guarantee the desired accuracy when calculating the roots of $f$ via $N$th Ehrlich-type method.
\end{itemize}

In the examples below, we use the function of initial conditions ${E_f \colon \mathcal{D} \to \Rset_+}$ defined by
\begin{equation} \label{eq:FIC3-SM-infty}
E_f(x) =  \left \| \frac{W_f(x)}{d(x)} \right \|_\infty \, ,
\end{equation}
where $W_f$ is the Weierstrass correction defined by \eqref{eq:Weierstrass-correction}.
We consider only the case ${p = \infty}$ since the other cases are similar.

Also, we use the real function $\alpha$ defined by
\begin{equation} \label{eq:alpha-infty}
 \alpha(t) = \frac{2}{1 - (n - 2) t + \sqrt{(1 - (n - 2) t)^2 - 4 t}} \, .
\end{equation}

It follows from Theorem~\ref{thm:semilocal-theorem-Ehrlich-second} that if there exists an integer ${m \ge 0}$ such that
\begin{equation}  \label{eq:example-initial-conditions-Ehrlich}
 E_f(x^{(m)}) \le \mathscr{R} = \frac{8}{(3 + \sqrt{8 n - 7})^2} \, ,
\end{equation}
then $f$ has only simple zeros and the Ehrlich-type iteration \eqref{eq:Ehrlich-high-order} is well-defined and converges to a root-vector
$\xi$ of $f$ with order of convergence ${2N + 1}$.
Besides, for all ${k \ge m}$ such that
\begin{equation} \label{eq:mu-alpha-infinity}
 E_f(x^{(k)}) < \mathscr{R} = \frac{8}{(3 + \sqrt{8 n - 7})^2}
\end{equation}
the following a posteriori error estimate holds:
\begin{equation}  \label{eq:example-posteriori-estimate}
\|x^{(k)} - \xi\|_\infty <\varepsilon_k,  \quad\text{where}\quad
\varepsilon_k = \alpha(E_f(x^{(k)})) \, \|W_f(x^{(k)})\|_\infty \, .
\end{equation}

In the examples, we apply the Ehrlich-type methods \eqref{eq:Ehrlich-high-order} for some ${N \ge 1}$ using the following stopping criterion:
\begin{equation}  \label{eq:stop-criterion}
\varepsilon_k < 10^{-15} \quad\text{and}\quad E_f(x^{(k)}) < \mathscr{R} \quad  (k \ge m).
\end{equation}

For given $N$ we calculate the smallest $m \ge 0$ which satisfies the convergence condition \eqref{eq:example-initial-conditions-Ehrlich},
the smallest $k \ge m$ for which the stopping criterion \eqref{eq:stop-criterion} is satisfied,
as well as the value of $\varepsilon_k$ for the last $k$.

In Table~2 the values of iterations are given to 15 decimal places. The values of other quantities
($\mathscr{R}$, $E_f(x^{(m)})$, etc.) are given to 6 decimal places.


\begin{example} \label{exmp:ZhPH-06}
We consider the polynomial
\[
{f(z) = z^{4} - 1}
\]
and the initial guess
\[
{x^{(0)} = (0.5 + 0.5i, - 1.36 + 0.42i, -0.25 + 1.28i, 0.46 - 1.37i)}
\]
which are taken from Zhang et al. \cite{ZPH06}.
We have ${\mathscr{R} = 0.125}$ and ${E(x^{(0)}) = 0.506619}$.
The results for this example are presented in Table~\ref{tab:ZhPH-example}.
For example, we can see that for ${N = 10}$ at the first iteration we have proved that
the Ehrlich-type method converges with order of convergence $21$ and that at the second iteration we have calculated the zeros $f$ with accuracy less than ${10^{-127}}$. Moreover, at the next iteration we obtain the zeros of $f$ with accuracy less than ${10^{-2682}}$.
Also, we can see that for ${N = 100}$ at the second iteration we have obtained the zeros of $f$
with accuracy less than ${10^{-11450}}$.

\begin{table}[!ht]
  \centering
  \caption{Values of $m$, $k$ and $\varepsilon_k$ for Example~\ref{exmp:ZhPH-06} (${\mathscr{R} = 0.125}$)}
  \label{tab:ZhPH-example}
  \begin{tabular}  {c c c c|c l l}
\hline
$N$ & $m$ & ${E_f(x^{(m)})}$ & $\varepsilon_m$ & $k$ & \multicolumn{1}{c}{$\varepsilon_k$} & \multicolumn{1}{c}{$\varepsilon_{k + 1}$}\\
\hline \\ [-2.0ex]
1 & $2$   & $0.010032$ & $1.457548 \times 10^{-2}$ & $4$ & $4.385760 \times 10^{-21}$  & $ 8.919073 \times 10^{-63}$\\
2 & $1$   & $0.067725$ & $1.242914 \times 10^{-1}$ & $3$ & $1.347060 \times 10^{-38}$  & $7.284576 \times 10^{-193}$\\
3 & $1$   & $0.015716$ & $2.300541 \times 10^{-2}$ & $3$ & $1.825502 \times 10^{-106}$ & $5.054741 \times 10^{-744}$\\
4 & $1$   & $0.002730$ & $3.887455 \times 10^{-3}$ & $2$ & $1.330837 \times 10^{-25}$  & $3.543773 \times 10^{-230}$\\
5 & $1$   & $0.001215$ & $1.722883 \times 10^{-3}$ & $2$ & $4.720064 \times 10^{-37}$  & $2.999643 \times 10^{-407}$\\
6 & $1$   & $0.000206$ & $2.927439 \times 10^{-4}$ & $2$ & $1.060096 \times 10^{-50}$  & $5.523501 \times 10^{-657}$\\
7 & $1$   & $0.000081$ & $1.155284 \times 10^{-4}$ & $2$ & $6.261239 \times 10^{-67}$  & $3.252761 \times 10^{-1002}$\\
8 & $1$   & $0.000014$ & $1.986052 \times 10^{-5}$ & $2$ & $6.080606 \times 10^{-85}$  & $3.570038 \times 10^{-1439}$\\
9 & $1$   & $0.000005$ & $7.910775 \times 10^{-6}$ & $2$ & $1.309022 \times 10^{-105}$ & $1.170454 \times 10^{-2002}$\\
10 & $1$  & $0.000000$ & $1.366899 \times 10^{-6}$ & $2$ & $4.301615 \times 10^{-128}$ & $8.477451 \times 10^{-2683}$\\
100 & $1$ & $0.000000$ & $\,\,1.820743 \times 10^{-57}$& $1$ & $1.820743 \times 10^{-57}$  & $3.460397 \times 10^{-11451}$\\
\hline
\end{tabular}
\end{table}
\FloatBarrier

In Table~\ref{tab:ZhPH-example-Weierstrass-numerical-results-N=10}, we present numerical results for Example~\ref{exmp:ZhPH-06}
in the case ${N = 10}$.

\begin{table}[!ht]
  \centering
  \caption{Numerical results for Example~\ref{exmp:ZhPH-06} in the case $N = 10$}
  \label{tab:ZhPH-example-Weierstrass-numerical-results-N=10}
  \begin{tabular} {l l l}
\hline
$k$ & \multicolumn{1}{c}{$x_1^{(k)}$} & \multicolumn{1}{c}{$x_2^{(k)}$} \\
\hline
$0$ & $0.5 \quad \qquad\qquad\qquad + 0.5i$    & $\text{--} 1.36\;\, \qquad\qquad\qquad + 0.42i$ \\
$1$ & $1.000000380419496 + 0.000000816235730i$ & $\text{--} 1.000000220051461 - 0.000000495915480i$\\
$2$ & $1.000000000000000 + 0.000000000000000i$ & $\text{--} 1.000000000000000 + 0.000000000000000i$\\
\hline
$k$ &\multicolumn{1}{c}{$x_3^{(k)}$}  & \multicolumn{1}{c}{$x_4^{(k)}$} \\
\hline
$0$ & $\text{--} 0.25\;\;\qquad\qquad\qquad + 1.28i$         & $\;\; 0.46 \;\, \qquad\qquad\qquad - 1.37i$\\
$1$ & $\;\; 0.000000277962637 + 0.999999578393062i$ & $     \text{--} 0.000000314533436 - 0.999998669784542i$\\
$2$ & $\;\; 0.000000000000000 + 1.000000000000000i$ & $\;\; 0.000000000000000 - 1.000000000000000i$\\
\hline
\end{tabular}
\end{table}
\FloatBarrier
\end{example}


\begin{example} \label{exmp:PetIlPet}
We consider the polynomial
\[
{f(z) = z^{15} + z^{14} + 1}
\]
and Aberth's initial approximation ${x^{(0)} \in \Cset^n}$ given by
(see Aberth \cite{Abe73} and Petkovi\'c et al. \cite{PIP07}):
\begin{equation} \label{eq:Abert-initial-guess}
x_\nu^{(0)} = - \frac{a_1}{n} + r_0 \exp{(i \theta_\nu)}, \quad \theta_\nu = \frac{\pi}{n} \left(2\nu - \frac{3}{2}\right),
\quad \nu = 1,\ldots,n,
\end{equation}
where ${a_1 = 1}$, ${r_0 = 2}$ and ${n = 15}$.
We have ${\mathscr{R} = 0.043061}$  and ${E(x^{(0)}) = 0.179999}$.
The results for this example are presented in Table~\ref{tab:PetIlPet-example}.
For example, we can see that for ${N = 30}$ at the third iteration we have obtained the zeros of $f$
with accuracy less than ${10^{-248}}$.
Moreover, at the next iteration we get the zeros of $f$ with accuracy less than ${10^{-15105}}$.

\begin{table}[!ht]
  \centering
  \caption{Values of $m$, $k$ and $\varepsilon_k$ for Example~\ref{exmp:PetIlPet} (${\mathscr{R} = 0.043061}$)}
  \label{tab:PetIlPet-example}
  \begin{tabular} {c c c c|c l l}
\hline
$N$ & $m$ & ${E_f(x^{(m)})}$ & $\varepsilon_m$ & $k$ & \multicolumn{1}{c}{$\varepsilon_k$} & \multicolumn{1}{c}{$\varepsilon_{k + 1}$}\\
\hline \\ [-2.0ex]
1 & $6$ & $0.036897$ & $3.187918 \times 10^{-2} $      & $9$ & $3.967908 \times 10^{-36}$ & $5.304009 \times 10^{-106}$\\
2 & $5$ & $0.000003$ & $1.182714 \times 10^{-6} $      & $6$ & $6.112531 \times 10^{-28}$ & $2.230412 \times 10^{-134}$\\
3 & $4$ & $0.000064$ & $2.475020 \times 10^{-5} $      & $5$ & $2.446120 \times 10^{-29}$ & $2.722168 \times 10^{-197}$\\
4 & $4$ & $0.000000$ & $\,\,1.550670 \times 10^{-11} $ & $5$ & $3.838741 \times 10^{-93}$ & $1.589981 \times 10^{-827}$\\
5 & $3$ & $0.005793$ & $2.415745 \times 10^{-3} $      & $4$ & $9.532339 \times 10^{-24}$ & $8.487351 \times 10^{-248}$\\
6 & $3$ & $0.000293$ & $1.127450 \times 10^{-4} $      & $4$ & $9.565008 \times 10^{-45}$ & $1.725858 \times 10^{-565}$\\
7 & $3$ & $0.000005$ & $2.173198 \times 10^{-6} $      & $4$ & $4.018844 \times 10^{-77}$ & $6.737932 \times 10^{-1138}$\\
8 & $3$ & $0.000000$ & $1.562375 \times 10^{-8} $      & $4$ & $1.162424 \times 10^{-123}$& $1.291370 \times 10^{-2080}$\\
9 & $3$ & $0.000000$ & $\,\,4.092421 \times 10^{-11} $ & $4$ & $4.245137 \times 10^{-187}$& $1.373908 \times 10^{-3530}$\\
10 & $3$ & $0.000000$ &$\,\,3.904607 \times 10^{-14} $ & $4$ & $4.643262 \times 10^{-270}$& $2.543247 \times 10^{-5644}$\\
30 & $2$ & $0.000055$ & $2.129417 \times 10^{-5} $     & $3$ & $5.721566 \times 10^{-249}$& $2.377023 \times 10^{-15106}$\\
\hline
\end{tabular}
\end{table}
\FloatBarrier

\end{example}


\begin{example} \label{exmp:Wilkinson-polynomial}
We consider the Wilkinson polynomial (\cite{Wil63})
\[
f(z) = \prod_{j = 1}^{20} (z - j) = z^{20} - 120 z^{14} + \ldots + 2 \; 432 \; 902 \; 008 \; 176 \; 640 \; 000 .
\]
and Abert's initial approximation \eqref{eq:Abert-initial-guess} with ${a_1 = - 120}$, ${r_0 = 20}$ and ${n = 20}$.
We have ${\mathscr{R} = 0.033867}$ and ${E(x^{(0)}) = 0.344409}$.
The results foe Example~\ref{exmp:Wilkinson-polynomial} are shown in Table~\ref{tab:Wilkinson-Aberth-example}.
For example, we for ${N = 100}$ at the seventh iteration we get the zeros of $f$
with accuracy less than ${10^{-13776}}$.

\begin{table}[!ht]
  \centering
  \caption{Values of $m$, $k$ and $\varepsilon_k$ for Example~\ref{exmp:Wilkinson-polynomial} (${\mathscr{R} = 0.033867}$ )}
  \label{tab:Wilkinson-Aberth-example}
  \begin{tabular} {c c c c|c l l}
\hline
$N$ & $m$ & ${E_f(x^{(m)})}$ & $\varepsilon_m$ & $k$ & \multicolumn{1}{c}{$\varepsilon_k$} & \multicolumn{1}{c}{$\varepsilon_{k + 1}$}\\
\hline \\ [-2.0ex]
1 & $18$ & $0.000060$ & $6.095859 \times 10^{-5} $     & $20$ & $1.620028 \times 10^{-38}$ & $4.276235 \times 10^{-114}$ \\
2 & $12$ & $0.015335$ & $2.153155 \times 10^{-2} $     & $14$ & $1.095084 \times 10^{-46}$ & $1.779476 \times 10^{-230}$ \\
3 & $10$ & $0.018005$ & $2.769333 \times 10^{-2} $     & $12$ & $8.917532 \times 10^{-86}$ & $4.482714 \times 10^{-596}$ \\
4 & $9$  & $0.005514$ & $6.130790 \times 10^{-3} $     & $10$ & $4.221856 \times 10^{-21}$ & $7.250879 \times 10^{-184}$ \\
5 & $9$  & $0.000000$ & $\,\,1.159694 \times 10^{-15}$ & $10$ & $5.021359 \times 10^{-165}$& $5.118016 \times 10^{-1808}$\\
6 & $8$  & $0.000237$ & $2.386016 \times 10^{-4} $     & $9$  &  $8.455240 \times 10^{-48}$ & $1.280870 \times 10^{-612}$ \\
7 & $8$  & $0.000000$ & $\,\,2.723047 \times 10^{-17}$ & $8$  &  $2.723047 \times 10^{-17}$ & $8.926059 \times 10^{-249}$ \\
8 & $7$  & $0.018995$ & $2.934241 \times 10^{-2} $     & $8$  &  $2.885374 \times 10^{-30}$ & $4.152134 \times 10^{-503}$ \\
9 & $7$  & $0.002180$ & $2.274734 \times 10^{-3} $     & $8$  &  $3.792876 \times 10^{-51}$ & $1.140751 \times 10^{-958}$ \\
10 & $7$ & $0.000000$ & $5.185525 \times 10^{-7} $     & $8$  &  $1.620086 \times 10^{-132}$ &$2.936276 \times 10^{-2768}$ \\
30 & $5$ & $0.000181$ & $1.821419 \times 10^{-4} $     & $6$  &  $1.395923 \times 10^{-226}$ &$1.902920 \times 10^{-13777}$ \\
\hline
\end{tabular}
\end{table}
\FloatBarrier

In the Figure~\ref{fig:Wilkinson-example-Ehrlich}, we present the trajectories of approximations generated
by the method \eqref{eq:Ehrlich-high-order} for ${N = 30}$ after $6$ iterations.


\begin{figure}[!ht]
\centering
\includegraphics[scale = 1.1]{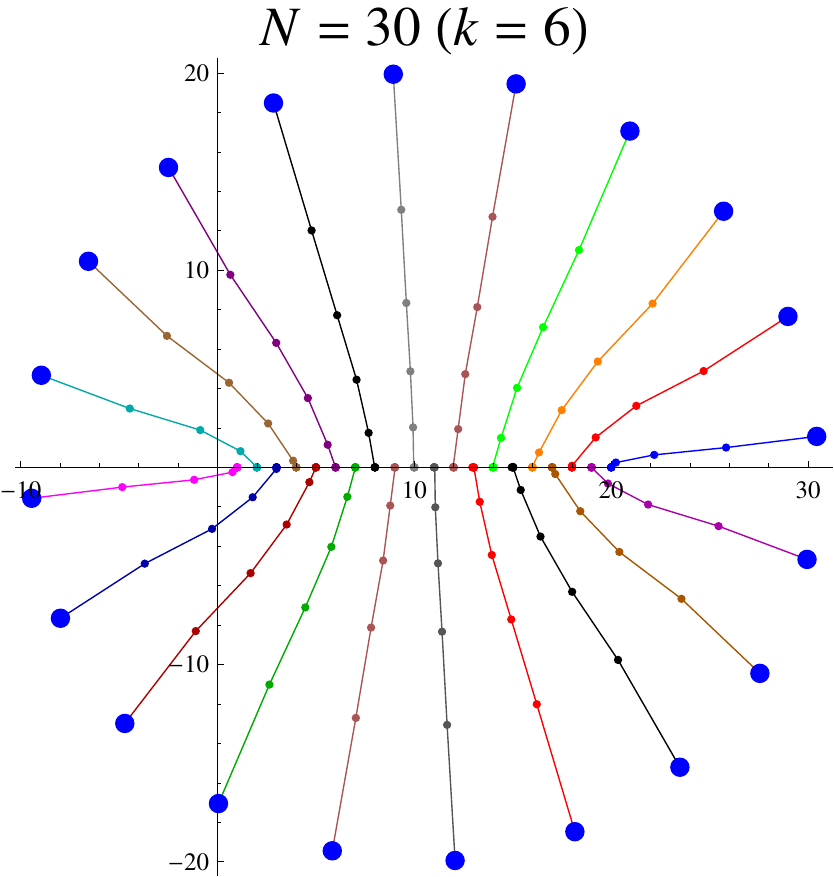}
\caption{Trajectories of approximations for the Wilkinson polynomial $f(z) = \prod_{j = 1}^{20} (z - j)$}
\label{fig:Wilkinson-example-Ehrlich}
\end{figure}
\FloatBarrier

\end{example}


\begin{example} \label{exmp:Wang-Zhao-40}
We consider the polynomial
\[
{f(z) = z^{40} - 1}.
\]
In this example we use Abert's initial approximation \eqref{eq:Abert-initial-guess} with ${a_1 = 0}$, ${r_0 = 2}$ and ${n = 40}$.
We have ${\mathscr{R} = 0.018685}$, ${E(x^{(0)}) = 0.159318}$.
The results for Example~\ref{exmp:Wang-Zhao-40} can be seen in Table~\ref{tab:Wang-Zhao-example-Abert-power-40}.

\begin{table}[!ht]
  \centering
  \caption{Values of $m$, $k$ and $\varepsilon_k$ for Example~\ref{exmp:Wang-Zhao-40} (${\mathscr{R} = 0.018685}$)}
  \label{tab:Wang-Zhao-example-Abert-power-40}
   \begin{tabular} {c c c c|c l l}
\hline
$N$ & $m$ & ${E_f(x^{(m)})}$ & $\varepsilon_m$ & $k$ & \multicolumn{1}{c}{$\varepsilon_k$} & \multicolumn{1}{c}{$\varepsilon_{k + 1}$}\\
\hline \\ [-2.0ex]
1 & $15$  & $0.007235$ & $1.588799 \times 10^{-3} $      & $17$ & $1.057241 \times 10^{-18}$  & $1.574672 \times 10^{-52}$\\
2 & $11$  & $0.000001$ & $1.731641 \times 10^{-7} $      & $12$ & $2.763909 \times 10^{-30}$  & $2.863869 \times 10^{-144}$\\
3 & $9$   & $0.000026$ & $4.171842 \times 10^{-6} $      & $10$ & $5.167701 \times 10^{-32}$  & $2.328540 \times 10^{-213}$\\
4 & $8$   & $0.000032$ & $5.141616 \times 10^{-6} $      & $9$  & $7.830010 \times 10^{-40}$  & $3.487627 \times 10^{-344}$\\
5 & $7$   & $0.010766$ & $2.954474 \times 10^{-3} $      & $8$  & $1.468181 \times 10^{-20}$  & $2.870206 \times 10^{-208}$\\
6 & $7$   & $0.000002$ & $4.201055 \times 10^{-7} $      & $8$  & $7.096655 \times 10^{-71}$  & $6.481892 \times 10^{-900}$\\
7 & $7$   & $0.000000$ & $\,\,9.445503 \times 10^{-15}$  & $8$  & $3.169914 \times 10^{-196}$ & $2.445585 \times 10^{-2918}$\\
8 & $6$   & $0.010675$ & $2.911647 \times 10^{-3} $      & $7$  & $8.218559 \times 10^{-31}$ & $3.538870 \times 10^{-495}$\\
9 & $6$   & $0.000281$ & $4.462548 \times 10^{-5} $      & $7$  & $2.324176 \times 10^{-64}$  & $1.205364 \times 10^{-1190}$\\
10 & $6$  & $0.000000$ & $1.231259 \times 10^{-7} $      & $7$  & $1.392265 \times 10^{-124}$ & $1.840079 \times 10^{-2580}$\\
30 & $5$  & $0.000000$ & $\,\,2.416285 \times 10^{-34}$  & $5$  & $2.416285 \times 10^{-34}$  & $1.294365 \times 10^{-1987}$\\
\hline
\end{tabular}
\end{table}
\FloatBarrier

In the Figure~\ref{fig:Wang-Zhao-example-Ehrlich}, we present the trajectories of approximations generated by the method
\eqref{eq:Ehrlich-high-order} for ${N = 30}$ after $5$ iterations.


\begin{figure}[!ht]
\centering
\includegraphics[scale = 1]{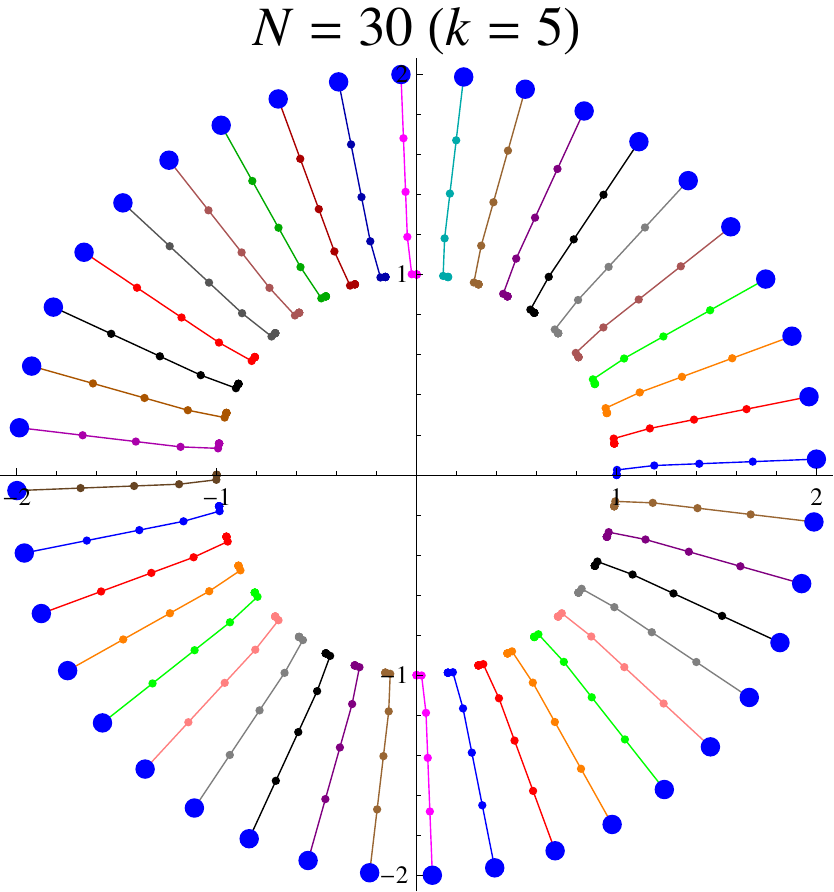}
\caption{Trajectories of approximations for the polynomial ${f(z) = z^{40} - 1}$}
\label{fig:Wang-Zhao-example-Ehrlich}
\end{figure}
\FloatBarrier

\end{example}


\begin{backmatter}

\section*{Competing interests}
 The authors declare that they have no competing interests.

\section*{Author's contributions}
    Both authors contributed equally and significantly in writing this paper. Both authors read and approved the final manuscript.

\section*{Acknowledgements}
 This research is supported by the project NI15-FMI-004 of Plovdiv University.





\newcommand{\BMCxmlcomment}[1]{}

\BMCxmlcomment{

<refgrp>

<bibl id="B1">
  <title><p>A unified theory of cone metric spaces and its applications to the
  fixed point theory</p></title>
  <aug>
    <au><snm>Proinov</snm><fnm>PD</fnm></au>
  </aug>
  <source>Fixed Point Theory Appl.</source>
  <pubdate>2013</pubdate>
  <volume>2013</volume>
  <fpage>ArticleID103</fpage>
</bibl>

<bibl id="B2">
  <title><p>Numerical {S}olution of {P}olynomial {E}quations</p></title>
  <aug>
    <au><snm>Sendov</snm><fnm>B</fnm></au>
    <au><snm>Andreev</snm><fnm>A.</fnm></au>
    <au><snm>Kjurkchiev</snm><fnm>N.</fnm></au>
  </aug>
  <source>Handbook of {N}umerical {A}nalysis</source>
  <publisher>Amsterdam: Elsevier</publisher>
  <pubdate>1994</pubdate>
  <volume>III</volume>
  <fpage>625</fpage>
  <lpage>-778</lpage>
</bibl>

<bibl id="B3">
  <title><p>Initial {A}pproximations and {R}oot {F}inding {M}ethods</p></title>
  <aug>
    <au><snm>Kyurkchiev</snm><fnm>N.V.</fnm></au>
  </aug>
  <publisher>Berlin: Wiley</publisher>
  <series><title><p>Mathematical Research</p></title></series>
  <pubdate>1998</pubdate>
  <volume>104</volume>
</bibl>

<bibl id="B4">
  <title><p>Numerical {M}ethods for {R}oots of {P}olynomials {P}art
  I</p></title>
  <aug>
    <au><snm>McNamee</snm><fnm>J.M.</fnm></au>
  </aug>
  <publisher>Amsterdam: Elsevier</publisher>
  <series><title><p>Studies in Computational Mathematics</p></title></series>
  <pubdate>2007</pubdate>
  <volume>14</volume>
</bibl>

<bibl id="B5">
  <title><p>Point {E}stimation of {R}oot {F}inding Methods</p></title>
  <aug>
    <au><snm>Petkovi\'c</snm><fnm>M</fnm></au>
  </aug>
  <publisher>Berlin: Springer</publisher>
  <series><title><p>Lecture Notes in Mathematics</p></title></series>
  <pubdate>2008</pubdate>
  <volume>1933</volume>
</bibl>

<bibl id="B6">
  <title><p>Neuer {B}eweis des {S}atzes, dass jede ganze rationale {F}unction
  einer {V}er\"anderlichen dargestellt werden kann als ein {P}roduct aus
  linearen {F}unctionen derselben {V}er\"anderlichen</p></title>
  <aug>
    <au><snm>Weierstrass</snm><fnm>K.</fnm></au>
  </aug>
  <source>Sitzungsber. K\"onigl. Akad. Wiss. Berlin</source>
  <pubdate>1891</pubdate>
  <fpage>1085</fpage>
  <lpage>-1101</lpage>
</bibl>

<bibl id="B7">
  <title><p>A modified {N}ewton method for polynomials</p></title>
  <aug>
    <au><snm>Ehrlich</snm><fnm>L.W.</fnm></au>
  </aug>
  <source>Comm. AMC</source>
  <pubdate>967</pubdate>
  <volume>10</volume>
  <issue>2</issue>
  <fpage>107</fpage>
  <lpage>-108</lpage>
</bibl>

<bibl id="B8">
  <title><p>Iteration methods for finding all zeros of a polynomial
  simultaneously</p></title>
  <aug>
    <au><snm>Abert</snm><fnm>O</fnm></au>
  </aug>
  <source>Math. Comput.</source>
  <pubdate>1973</pubdate>
  <volume>27</volume>
  <fpage>339</fpage>
  <lpage>-344</lpage>
</bibl>

<bibl id="B9">
  <title><p>Residuenabschatzung fur {P}olynom-{N}ullstellen mittels
  {L}agrange-{I}nterpolation</p></title>
  <aug>
    <au><snm>B\"orsch Supan</snm><fnm>W.</fnm></au>
  </aug>
  <source>Numer. Math</source>
  <pubdate>1970</pubdate>
  <volume>14</volume>
  <fpage>287</fpage>
  <lpage>-296</lpage>
</bibl>

<bibl id="B10">
  <title><p>On the simultaneous determination of polynomial roots</p></title>
  <aug>
    <au><snm>Werner</snm><fnm>W.</fnm></au>
  </aug>
  <source>Lecture Notes Math.</source>
  <pubdate>1982</pubdate>
  <volume>953</volume>
  <fpage>188</fpage>
  <lpage>-202</lpage>
</bibl>

<bibl id="B11">
  <title><p>On the local convergence of the {E}hrlich method for numerical
  computation of polynomial zeros</p></title>
  <aug>
    <au><snm>Proinov</snm><fnm>PD</fnm></au>
  </aug>
  <note>submitted</note>
</bibl>

<bibl id="B12">
  <title><p>A method for simultaneous determination of all roots of algebraic
  polynomials</p></title>
  <aug>
    <au><snm>Kyurkchiev</snm><fnm>N.V.</fnm></au>
    <au><snm>Taschev</snm><fnm>S.</fnm></au>
  </aug>
  <source>C.~R. Acad. Bulg. Sci.</source>
  <pubdate>1981</pubdate>
  <volume>34</volume>
  <fpage>1053</fpage>
  <lpage>-1055</lpage>
  <note>in Russian</note>
</bibl>

<bibl id="B13">
  <title><p>Certain modifications of {N}ewton's method for the approximate
  solution of algebraic equations</p></title>
  <aug>
    <au><snm>Tashev</snm><fnm>S.</fnm></au>
    <au><snm>Kyurkchiev</snm><fnm>N.</fnm></au>
  </aug>
  <source>Serdica Math. J.</source>
  <pubdate>1983</pubdate>
  <volume>9</volume>
  <fpage>67</fpage>
  <lpage>-73</lpage>
  <note>in Russian</note>
</bibl>

<bibl id="B14">
  <title><p>Complexity analysis of a process for simultaneously obtaining all
  zeros of polynomials</p></title>
  <aug>
    <au><snm>Wang</snm><fnm>D.R.</fnm></au>
    <au><snm>Zhao</snm><fnm>F.G.</fnm></au>
  </aug>
  <source>Computing</source>
  <pubdate>1989</pubdate>
  <volume>43</volume>
  <fpage>187</fpage>
  <lpage>-197</lpage>
</bibl>

<bibl id="B15">
  <title><p>Convergence conditions of some methods for the simultaneous
  computation of polynomial zeros</p></title>
  <aug>
    <au><snm>Tilli</snm><fnm>P.</fnm></au>
  </aug>
  <source>Calcolo</source>
  <pubdate>1998</pubdate>
  <volume>35</volume>
  <fpage>3</fpage>
  <lpage>-15</lpage>
</bibl>

<bibl id="B16">
  <title><p>Ehrlich's methods with a raised speed of convergence</p></title>
  <aug>
    <au><snm>Kjurkchiev</snm><fnm>N.V.</fnm></au>
    <au><snm>Andreev</snm><fnm>A.</fnm></au>
  </aug>
  <source>Serdica Math. J.</source>
  <pubdate>1987</pubdate>
  <volume>13</volume>
  <fpage>52</fpage>
  <lpage>-57</lpage>
</bibl>

<bibl id="B17">
  <title><p>General local convergence theory for a class of iterative processes
  and its applications to {N}ewton's method</p></title>
  <aug>
    <au><snm>Proinov</snm><fnm>PD</fnm></au>
  </aug>
  <source>J. Complexity</source>
  <pubdate>2009</pubdate>
  <volume>25</volume>
  <fpage>38</fpage>
  <lpage>-62</lpage>
</bibl>

<bibl id="B18">
  <title><p>New general convergence theory for iterative processes and its
  applications to {N}ewton-{K}antorovich type theorems</p></title>
  <aug>
    <au><snm>Proinov</snm><fnm>PD</fnm></au>
  </aug>
  <source>J. Complexity</source>
  <pubdate>2010</pubdate>
  <volume>26</volume>
  <fpage>3</fpage>
  <lpage>-42</lpage>
</bibl>

<bibl id="B19">
  <title><p>General convergence theorems for iterative processes and
  applications to the {W}eierstrass root-finding method</p></title>
  <aug>
    <au><snm>Proinov</snm><fnm>PD</fnm></au>
  </aug>
  <source>arXiv:1503.05243</source>
  <pubdate>2015</pubdate>
</bibl>

<bibl id="B20">
  <title><p>Semilocal convergence of {C}hebyshev-like root-finding method for
  simultaneous approximation of polynomial zeros</p></title>
  <aug>
    <au><snm>Proinov</snm><fnm>P.D.</fnm></au>
    <au><snm>Cholakov</snm><fnm>S.I.</fnm></au>
  </aug>
  <source>Appl. Math. Comput.</source>
  <pubdate>2014</pubdate>
  <volume>236</volume>
  <fpage>669</fpage>
  <lpage>-682</lpage>
</bibl>

<bibl id="B21">
  <title><p>On the convergence of a family of {W}eierstrass-type root-finding
  methods</p></title>
  <aug>
    <au><snm>Proinov</snm><fnm>P.D.</fnm></au>
    <au><snm>Vasileva</snm><fnm>M.T.</fnm></au>
  </aug>
  <source>C.~R. Acad. Bulg. Sci.</source>
  <pubdate>2015</pubdate>
  <volume>68</volume>
  <fpage>697</fpage>
  <lpage>-704</lpage>
</bibl>

<bibl id="B22">
  <title><p>Relationships between different types of initial conditions for
  simultaneous root finding methods</p></title>
  <aug>
    <au><snm>Proinov</snm><fnm>PD</fnm></au>
  </aug>
  <source>arXiv:1506.01043</source>
  <pubdate>2015</pubdate>
</bibl>

<bibl id="B23">
  <title><p>A high order iteration formula for the simultaneous inclusion of
  polynomial zeros</p></title>
  <aug>
    <au><snm>Zhang</snm><fnm>X</fnm></au>
    <au><snm>Peng</snm><fnm>H</fnm></au>
    <au><snm>Hu</snm><fnm>G</fnm></au>
  </aug>
  <source>Appl. Math. Comput.</source>
  <pubdate>2006</pubdate>
  <volume>179</volume>
  <fpage>545</fpage>
  <lpage>–552</lpage>
</bibl>

<bibl id="B24">
  <title><p>A posteriori error bound methods for the inclusion of polynomial
  zeros</p></title>
  <aug>
    <au><snm>Petkovi\'c</snm><fnm>M.</fnm></au>
    <au><snm>Ili\'c</snm><fnm>S.</fnm></au>
    <au><snm>Petkovi\'c</snm><fnm>I.</fnm></au>
  </aug>
  <source>J. Comput. Appl. Math.</source>
  <pubdate>2007</pubdate>
  <volume>208</volume>
  <fpage>316</fpage>
  <lpage>-330</lpage>
</bibl>

<bibl id="B25">
  <title><p>Rounding {E}rrors in {A}lgebraic {P}rocesses</p></title>
  <aug>
    <au><snm>Wilkinson</snm><fnm>J.H.</fnm></au>
  </aug>
  <publisher>New Jersey: Prentice Hall</publisher>
  <pubdate>1963</pubdate>
</bibl>

</refgrp>
} 

\end{backmatter}

\end{document}